\documentclass[11pt]{article}

\usepackage{amsmath, amssymb, amsthm} 
\usepackage{stmaryrd} 
\usepackage{thmtools} 
\usepackage{graphicx}
\usepackage{xspace}
\allowdisplaybreaks
\expandafter\let\expandafter\oldproof\csname\string\proof\endcsname
\let\oldendproof\endproof
\renewenvironment{proof}[1][\proofname]{%
	\oldproof[\bf #1]%
}{\oldendproof}

\allowdisplaybreaks

\theoremstyle{plain}
\newtheorem{theorem}{Theorem}[section]
\newtheorem{lemma}[theorem]{Lemma}
\newtheorem{claim}[theorem]{Claim}
\newtheorem{proposition}[theorem]{Proposition}
\newtheorem{observation}[theorem]{Observation}
\newtheorem{corollary}[theorem]{Corollary}

\newtheorem{remark}[theorem]{Remark}

\newcommand{\C}{\mathcal C}

\newcommand{\Z}{\mathcal Z}
\renewcommand{\P}{\mathcal P}
\newcommand{\T}{\mathcal{T}}
\newcommand{\K}{\mathcal{K}}

\usepackage{mathrsfs}

\newcommand{\sG}{\mathscr{G}}
\newcommand{\HH}{\mathcal{H}}
\newcommand{\sS}{\mathcal{S}}

\DeclareMathOperator{\Pro}{\mathbb{P}}

\DeclareMathOperator{\Ex}{\mathbb{E}}

\newcommand{\floor}[1]{\left\lfloor #1 \right\rfloor}
\newcommand{\ceil}[1]{\left\lceil #1 \right\rceil}
\newcommand{\Gnp}{\mathbb{G}(n,p)}
\def\rainbow{\stackrel{\mathrm{rbw}}{\longrightarrow}}
\def\notrainbow{\stackrel{\mathrm{rbw}}{\longarrownot\longrightarrow}}

\makeatletter
\def\moverlay{\mathpalette\mov@rlay}
\def\mov@rlay#1#2{\leavevmode\vtop{%
   \baselineskip\z@skip \lineskiplimit-\maxdimen
   \ialign{\hfil$\m@th#1##$\hfil\cr#2\crcr}}}
\newcommand{\charfusion}[3][\mathord]{
    #1{\ifx#1\mathop\vphantom{#2}\fi
        \mathpalette\mov@rlay{#2\cr#3}
      }
    \ifx#1\mathop\expandafter\displaylimits\fi}
\makeatother

\newcommand{\discup}{\charfusion[\mathbin]{\cup}{\cdot}}

\let\eps=\varepsilon
\let\theta=\vartheta
\let\rho=\varrho
\let\phi=\varphi

\usepackage{framed}

\usepackage{enumitem}
\usepackage{subcaption}
\usepackage{color}
\usepackage[table]{xcolor}
\usepackage{caption}
\usepackage[bookmarks = false]{hyperref}

\makeatletter
\renewcommand*{\eqref}[1]{%
  \hyperref[{#1}]{\textup{\tagform@{\ref*{#1}}}}%
}
\makeatother

\usepackage{tikz}
\usetikzlibrary{calc}
\usetikzlibrary{patterns}
\usetikzlibrary{arrows,decorations.pathmorphing,backgrounds,positioning,fit,petri}
\usetikzlibrary{decorations.text}
\usetikzlibrary{matrix}

\definecolor{myMaroon}{HTML}{720E0E}
\definecolor{myBlue}{HTML}{1A5276}
\definecolor{myDarkerBlue}{HTML}{154360}
\definecolor{my_blue}{HTML}{4A5DE2}
\definecolor{my_purple}{HTML}{9013FE}
\definecolor{my_red}{HTML}{D0021B}
\definecolor{my_cyan}{HTML}{2CA78B}
\definecolor{my_green}{HTML}{417505}
\definecolor{my_yellow}{HTML}{F5A623}

\definecolor{1}{HTML}{000000}
\definecolor{2}{HTML}{000000}
\definecolor{6}{HTML}{000000}
\definecolor{4}{HTML}{000000}
\definecolor{7}{HTML}{000000}
\definecolor{3}{HTML}{000000}
\definecolor{5}{HTML}{000000}

\tikzset{colour1/.style={
        color = 1,
    }
}
\tikzset{colour1/.style={
        color = 1,
    }
}\tikzset{colour2/.style={
        color = 2,
    }
}\tikzset{colour3/.style={
        color = 3,
    }
}\tikzset{colour4/.style={
        color = 4,
        densely dotted
    }
}\tikzset{colour5/.style={
        color = 5,
        densely dashed
    }
}\tikzset{colour6/.style={
        color = 6,
        densely dashdotted
    }
}\tikzset{colour7/.style={
        color = 7,
        dash pattern=on 3mm off 1mm
    }
}

\newcommand{\KvsK}[1]{%
    \begin{tikzpicture}
        [inner sep=2mm,
        xNode/.style={circle,draw=blue!50,fill=blue!20,thick,inner sep=0pt,minimum size=6mm},
        blackNode/.style={circle,draw=black,text=white,fill=black,thick,inner sep=0pt,minimum size=6mm},
        yNode/.style={rectangle,draw=black!50,fill=black!20,thick,inner sep=0pt,minimum size=4mm},
        1/.style={rectangle,draw=black!50,fill=black!20,thick,inner sep=0pt,minimum size=4mm},
        #1]
            \node[blackNode]      (x1)       at ( 0,4) {$x_1$};
            \node[blackNode]      (x2)       at ( 0,2) {$x_2$};
            \node[blackNode]      (x3)       at ( 0,0) {$x_3$};
            \node[blackNode]      (x1')      at ( 4,4) {$x'_1$};
            \node[blackNode]      (x2')      at ( 4,2) {$x'_2$};
            \node[blackNode]      (x3')      at ( 4,0) {$x'_3$};
            \node[blackNode]    (y)        at (-2,2) {$y$};
            \node[blackNode]    (y')       at (6,2)  {$y'$};
            \draw[colour2] (x2) -- (y)   node [midway]  {2};
            \draw[colour3] (y) -- (x3)   node [midway, below,outer sep=-3pt]  {3};
            \draw[colour1] (y) -- (x1)   node [midway, above,outer sep=-3pt]  {1};
            \draw[colour2] (x2') -- (y') node [midway]  {2};
            \draw[colour3] (y') -- (x3') node [midway, below,outer sep=-3pt]  {3};
            \draw[colour1] (y') -- (x1') node [midway, above,outer sep=-3pt]  {1};
            \draw[colour4]   (x1) --  (x2');
            \draw[colour4]   (y) to [out=40,in=140] (y');
            \draw[colour4]   (x2) -- (x3');
            \draw[colour4]   (x3) -- (x1');
            \draw[colour5] (y) to [out=17,in=165] (x2');
            \draw[colour5] (x3) -- (y');
            \draw[colour6] (x1) -- (y');
            \draw[colour6] (y) -- (x3');
            \draw[colour7] (x2) to [out=40,in=150] (y');
            \draw[colour7] (y) -- (x1');

            \matrix [below right] at (7,5) {
  			\draw[colour4] (0,0)--(0.7,0) node[pos=1.1,right] {{\small $4$}}; \\
 			\draw[colour5] (0,-.1)--(0.7,-.1) node[pos=1.1,right] {{\small $5$}}; \\
  			\draw[colour6] (0,-.2)--(0.7,-.2) node[pos=1.1,right] {{\small $6$}}; \\
  			\draw[colour7] (0,-.3)--(0.7,-.3) node[pos=1.1,right] {{\small $7$}};  \\
 			};
            
    \end{tikzpicture}
}

\newcommand{\KvsPFour}[1]
{
    \begin{tikzpicture}
        [inner sep=2mm,
        xNode/.style={circle,draw=blue!50,fill=blue!20,thick,inner sep=0pt,minimum size=6mm},
        blackNode/.style={circle,draw=black,text=white,fill=black,thick,inner sep=0pt,minimum size=6mm},
        yNode/.style={rectangle,draw=black!50,fill=black!20,thick,inner sep=0pt,minimum size=4mm},
        1/.style={rectangle,draw=black!50,fill=black!20,thick,inner sep=0pt,minimum size=4mm},#1]
        \foreach \i in {1,2,3}
            \node[blackNode] (x\i) at (1,7-3*\i) {$x_\i$};
        \foreach \i in {1,2,3,4} 
            \node[blackNode] (x\i') at (4,6-2*\i) {$x'_{\i}$}; 
        \node[blackNode] (y) at (-2,1) {$y$};
    
        \draw[colour1] (y) -- (x1) node [midway, above] {1};
        \draw[colour1] (y) -- (x2) node [pos=.4] {2};
        \draw[colour1] (y) -- (x3) node [midway, below] {3};
        
        \draw[colour1] (x1') -- (x2') node [midway, right,outer sep=-3pt]  {1};;
        \draw[colour2] (x2') -- (x3') node [midway, right,outer sep=-3pt]  {2};;
        \draw[colour3] (x3') -- (x4') node [midway, right,outer sep=-3pt]  {3};;
        \draw[colour4] (y)   -- (x1') ;
        \draw[colour4] (x2)  -- (x2') ;
        \draw[colour5] (y)   -- (x4') ;
        \draw[colour5] (x2)  -- (x3') ;
        \draw[colour6] (x1)  -- (x3') ;
        \draw[colour6] (y) -- (x2')   ;
        \draw[colour6] (x3) -- (x1')  ;
        \draw[colour7] (x1) -- (x4')  ;
        \draw[colour7] (y) -- (x3')   ;
        \draw[colour7] (x3) -- (x2')  ;

            \matrix [below right] at (7,5) {
  			\draw[colour4] (0,0)--(0.7,0) node[pos=1.1,right] {{\small $4$}}; \\
 			\draw[colour5] (0,-.1)--(0.7,-.1) node[pos=1.1,right] {{\small $5$}}; \\
  			\draw[colour6] (0,-.2)--(0.7,-.2) node[pos=1.1,right] {{\small $6$}}; \\
  			\draw[colour7] (0,-.3)--(0.7,-.3) node[pos=1.1,right] {{\small $7$}};  \\
 			};

    \end{tikzpicture}
}

\newcommand{\PFourVsPFour}[1]
{
    \begin{tikzpicture}
        [inner sep=2mm,
        xNode/.style={circle,draw=blue!50,fill=blue!20,thick,inner sep=0pt,minimum size=6mm},
        blackNode/.style={circle,draw=black,text=white,fill=black,thick,inner sep=0pt,minimum size=6mm},
        yNode/.style={rectangle,draw=black!50,fill=black!20,thick,inner sep=0pt,minimum size=4mm},
        1/.style={rectangle,draw=black!50,fill=black!20,thick,inner sep=0pt,minimum size=4mm},#1]
        textNode={outer sep=-3pt}
        \foreach \i in {1,2,3,4}
            \node[blackNode] (x\i) at (0,6-2*\i) {$x_\i$};
        \foreach \i in {1,2,3,4} 
            \node[blackNode] (x\i') at (4,6-2*\i) {$x'_{\i}$}; 
    
        \draw[colour1] (x1) --  (x2)   node[midway, left,outer sep=-3pt] {1};
        \draw[colour2] (x2) --  (x3)   node[midway, left,outer sep=-3pt] {2};
        \draw[colour3] (x3) --  (x4)   node[midway, left,outer sep=-3pt] {3};
        \draw[colour1] (x1') -- (x2')  node[midway, right,outer sep=-3pt]{1};
        \draw[colour2] (x2') -- (x3')  node[midway, right,outer sep=-3pt]{2};
        \draw[colour3] (x3') -- (x4')  node[midway, right,outer sep=-3pt]{3};
        \draw[colour4] (x1)   -- (x3')  ;
        \draw[colour4] (x2)   --  (x4') ;
        \draw[colour4] (x3)   -- (x1')  ;
        \draw[colour4] (x4)   -- (x2')  ;
        \draw[colour5] (x1)  -- (x2')   ;
        \draw[colour5] (x2)  -- (x3')   ;
        \draw[colour5] (x3) -- (x4')    ;
        \draw[colour6] (x2) -- (x1')    ;
        \draw[colour6] (x3) -- (x2')    ;
        \draw[colour6] (x4) -- (x3')    ;

            \matrix [below right] at (7,5) {
  			\draw[colour4] (0,0)--(0.7,0) node[pos=1.1,right] {{\small $4$}}; \\
 			\draw[colour5] (0,-.1)--(0.7,-.1) node[pos=1.1,right] {{\small $5$}}; \\
  			\draw[colour6] (0,-.2)--(0.7,-.2) node[pos=1.1,right] {{\small $6$}}; \\
 			};
        
    \end{tikzpicture}
}

\def \vx {circle [radius = .07][fill = black]}

\newcommand \defPt[3]{
    \def \pt {(#1, #2)}
    \coordinate [at = \pt, name = #3];
}

\def\triangleedge#1#2#3#4#5{%
    \path
        coordinate (a) at #3
        coordinate (b) at #4
        coordinate (c) at #5;

    \path
        coordinate (a1) at ($#3!#2!-90:#4$)
        coordinate (a2) at ($#4!#2!90:#3$)
        coordinate (b1) at ($#4!#2!-90:#5$)
        coordinate (b2) at ($#5!#2!90:#4$)
        coordinate (c1) at ($#5!#2!-90:#3$)
        coordinate (c2) at ($#3!#2!90:#5$);

    \pgfmathanglebetweenpoints{%
        \pgfpointanchor{a}{center}}{%
        \pgfpointanchor{c2}{center}
    }
    \edef\angleCA{\pgfmathresult}
    
    \pgfmathanglebetweenpoints{%
        \pgfpointanchor{a}{center}}{%
        \pgfpointanchor{a1}{center}
    }
    \edef\angleAC{\pgfmathresult}
    
    \pgfmathanglebetweenpoints{%
        \pgfpointanchor{b}{center}}{%
        \pgfpointanchor{a2}{center}
    }
    \edef\angleAB{\pgfmathresult}
    
    \pgfmathanglebetweenpoints{%
        \pgfpointanchor{b}{center}}{%
        \pgfpointanchor{b1}{center}
    }
    \edef\angleBA{\pgfmathresult}
    
    \pgfmathanglebetweenpoints{%
        \pgfpointanchor{c}{center}}{%
        \pgfpointanchor{b2}{center}
    }
    \edef\angleBC{\pgfmathresult}
    
    \pgfmathanglebetweenpoints{%
        \pgfpointanchor{c}{center}}{%
        \pgfpointanchor{c1}{center}
    }
    \edef\angleCB{\pgfmathresult}
    
    \ifdim \angleAB pt > \angleBA pt
        \edef\angleAB{\angleAB-360}
    \fi

    \ifdim \angleBC pt > \angleCB pt
        \edef\angleBC{\angleBC-360}
    \fi

    \ifdim \angleCA pt > \angleAC pt
        \edef\angleCA{\angleCA-360}
    \fi

    \draw[#1]
        (a1) -- (a2) arc (\angleAB:\angleBA:#2)
             -- (b2) arc (\angleBC:\angleCB:#2)
             -- (c2) arc (\angleCA:\angleAC:#2) -- cycle;
}

\tikzstyle {edge} = [very thick]

\newcommand{\figone}[1]
{
    \begin{tikzpicture}[#1]
   	\def \r{.7}
	\def \s{.7}

	\defPt{0}{0}{A1}
	\defPt{\r}{0}{A2}

	\defPt{-\r}{0}{A3}
	\defPt{.5*\r}{\s}{A4}
	\defPt{-.5*\r}{\s}{A5}

	\foreach \i in {1,...,5}
		\draw (A\i) \vx;

	\draw[very thick] (A4) -- (A2) -- (A1) -- (A3) -- (A5) -- (A1) -- (A4) -- (A5);

    \end{tikzpicture}
}

\newcommand{\figtwo}[1]
{
\begin{tikzpicture}[#1]
	
	\def \r{.8}
	\def \s{.7}

	\defPt{0}{0}{A1}
	\defPt{-.5*\r}{.5*\s}{A2}
	\defPt{0}{\s}{A3}
	\defPt{.5*\r}{.5*\s}{A4}

	\foreach \i in {1,...,4}
		\draw (A\i) \vx;

	\draw[very thick] (A1) -- (A2) -- (A3) -- (A1) -- (A4) -- (A3);

\end{tikzpicture}

}

\newcommand{\figthree}[1]
{
\begin{tikzpicture}[#1, line join = bevel]
	
	\def \r{.7}
	\def \s{.7}

	\defPt{0}{0}{A1}
	\defPt{\r}{0}{A2}

	\defPt{-\r}{0}{A3}
	\defPt{.5*\r}{\s}{A4}
	\defPt{-.5*\r}{\s}{A5}

	\foreach \i in {1,...,5}
		\draw (A\i) \vx;

	\draw[very thick] (A1) -- (A4) -- (A2) -- (A5) -- (A3) -- (A4) -- (A5) -- (A1);

\end{tikzpicture}

}

\newcommand{\figfour}[1]
{
\begin{tikzpicture}[#1]

	\def \r{.7}
	\def \s{.4}

	\def \m{2}
	\def \n{6}
	\def \d{.5*\r}
	\foreach \i in {0,...,\m}
	{
		\defPt{\i*\r}{0}{A\i}
		\ifthenelse{\i = 0}{}{
			\defPt{\i*\r - .5*\r}{\s}{D\i}
		}
		\pgfmathsetmacro{\j}{\n-\i}
		\defPt{\j*\r}{0}{B\i}
		\ifthenelse{\i = 0}{}{
			\defPt{\j*\r + .5*\r}{\s}{E\i}
		}
	}
	\defPt{.5*\r}{-\s}{D0}
	\defPt{\n*\r - .5*\r}{-\s}{E0}

	\foreach \i in {0,...,\m}
	{
		\draw (A\i) \vx;
		\draw (D\i) \vx;
		\draw (B\i) \vx;
		\draw (E\i) \vx;
	}

	\foreach \i in {1,...,\m}
	{
		\pgfmathsetmacro{\j}{\i-1}
		\draw[very thick] (A\j) -- (D\i) -- (A\i);
		\pgfmathsetmacro{\j}{\i-1}
		\draw[very thick] (B\j) -- (E\i) -- (B\i);
	}

	\draw [dash pattern=on \pgflinewidth off 3pt, line width = .7pt] ($(A\m) + (\d, 0)$) -- ($(B\m) + (-\d, 0)$);

	\draw[very thick] (A0) -- (D1) -- (A1) -- (D0) -- (A0);
	\draw[very thick] (B0) -- (E1) -- (B1) -- (E0) -- (B0);
	
	\draw[very thick] (A1) -- (A\m);
	\draw[very thick] (B1) -- (B\m);

	\draw[very thick] (D0) -- (D1);
	\draw[very thick] (E0) -- (E1);

\end{tikzpicture}

}

\newcommand{\figfoura}[1]
{
\begin{tikzpicture}[#1]
	\def \r{.7}
	\def \s{.4}

	\def \m{2}
	\def \n{6}
	\def \d{.5*\r}
	\foreach \i in {0,...,\m}
	{
		\defPt{\i*\r}{0}{A\i}
		\ifthenelse{\i = 0}{}{
			\defPt{\i*\r - .5*\r}{\s}{D\i}
		}
		\pgfmathsetmacro{\j}{\n-\i}
		\defPt{\j*\r}{0}{B\i}
		\ifthenelse{\i = 0}{}{
			\defPt{\j*\r + .5*\r}{\s}{E\i}
		}
	}
	\defPt{.5*\r}{-\s}{D0}
	\defPt{\n*\r - .5*\r}{-\s}{E0}

	\foreach \i in {1,...,\m}
	{
		\pgfmathsetmacro{\j}{\i-1}
		\draw[thick] (A\j) -- (D\i);
		\ifthenelse{\i=1}{
			\draw[thick] (D\i) -- (A\i);
		}{
			\draw[ultra thick, red] (D\i) -- (A\i);
		}
		\pgfmathsetmacro{\j}{\i-1}
		\ifthenelse{\i=1}{
			\draw[thick] (B\j) -- (E\i);
		}{
			\draw[ultra thick, red] (B\j) -- (E\i);
		}
		\draw[thick] (E\i) -- (B\i);
	}

	\draw [dash pattern=on \pgflinewidth off 3pt, line width = .7pt] ($(A\m) + (\d, 0)$) -- ($(B\m) + (-\d, 0)$);

	\draw[thick] (A0) -- (D1) -- (A1) -- (D0) -- (A0);
	\draw[thick] (B0) -- (E1) -- (B1) -- (E0) -- (B0);
	
	\draw[thick] (A1) -- (A\m);
	\draw[thick] (B1) -- (B\m);

	\draw[ultra thick, red] (D0) -- (D1);
	\draw[ultra thick, red] (E0) -- (E1);

	\foreach \i in {0,...,\m}
	{
		\draw (A\i) \vx;
		\draw (D\i) \vx;
		\draw (B\i) \vx;
		\draw (E\i) \vx;
	}

\end{tikzpicture}
}

\newcommand{\figfive}[1]
{
\begin{tikzpicture}[#1]
	
	\def \r{.7}
	\def \s{.4}

	\def \m{2}
	\def \n{6}
	\def \d{.5*\r}
	\foreach \i in {0,...,\m}
	{
		\defPt{\i*\r}{0}{A\i}
		\ifthenelse{\i = 0}{}{
			\defPt{\i*\r - .5*\r}{\s}{D\i}
		}
		\pgfmathsetmacro{\j}{\n-\i}
		\defPt{\j*\r}{0}{B\i}
		\ifthenelse{\i = 0}{}{
			\defPt{\j*\r + .5*\r}{\s}{E\i}
		}
	}
	\defPt{.5*\r}{-\s}{D0}
	\defPt{\n*\r - .5*\r}{-\s}{E0}

	\foreach \i in {0,...,\m}
	{
		\draw (A\i) \vx;
		\draw (D\i) \vx;
		\draw (B\i) \vx;
		\draw (E\i) \vx;
	}

	\foreach \i in {1,...,\m}
	{
		\pgfmathsetmacro{\j}{\i-1}
		\draw[very thick] (A\j) -- (D\i) -- (A\i);
		\pgfmathsetmacro{\j}{\i-1}
		\draw[very thick] (B\j) -- (E\i) -- (B\i);
	}

	\draw [dash pattern=on \pgflinewidth off 3pt, line width = .7pt] ($(A\m) + (\d, 0)$) -- ($(B\m) + (-\d, 0)$);

	\draw[very thick] (A0) -- (D1) -- (A1) -- (D0) -- (A0);
	\draw[very thick] (B0) -- (E1) -- (B1) -- (E0) -- (B0);
	
	\draw[very thick] (A1) -- (A\m);
	\draw[very thick] (B1) -- (B\m);

	\draw[very thick] (A0) -- (A1);
	\draw[very thick] (E0) -- (E1);

\end{tikzpicture}

}

\newcommand{\figfivea}[1]
{
\begin{tikzpicture}[#1]

	\def \r{.7}
	\def \s{.4}

	\def \m{2}
	\def \n{6}
	\def \d{.5*\r}
	\foreach \i in {0,...,\m}
	{
		\defPt{\i*\r}{0}{A\i}
		\ifthenelse{\i = 0}{}{
			\defPt{\i*\r - .5*\r}{\s}{D\i}
		}
		\pgfmathsetmacro{\j}{\n-\i}
		\defPt{\j*\r}{0}{B\i}
		\ifthenelse{\i = 0}{}{
			\defPt{\j*\r + .5*\r}{\s}{E\i}
		}
	}
	\defPt{.5*\r}{-\s}{D0}
	\defPt{\n*\r - .5*\r}{-\s}{E0}

	\foreach \i in {1,...,\m}
	{
		\pgfmathsetmacro{\j}{\i-1}
		\draw[thick] (A\j) -- (D\i);
		\ifthenelse{\i=1}{
			\draw[thick] (D\i) -- (A\i);
		}{
			\draw[ultra thick, red] (D\i) -- (A\i);
		}
		\pgfmathsetmacro{\j}{\i-1}
		\ifthenelse{\i=1}{
			\draw[thick] (B\j) -- (E\i);
		}{
			\draw[ultra thick, red] (B\j) -- (E\i);
		}
		\draw[thick] (E\i) -- (B\i);
	}

	\draw [dash pattern=on \pgflinewidth off 3pt, line width = .7pt] ($(A\m) + (\d, 0)$) -- ($(B\m) + (-\d, 0)$);

	\draw[thick] (A0) -- (D1) -- (A1) -- (D0) -- (A0);
	\draw[thick] (B0) -- (E1) -- (B1) -- (E0) -- (B0);
	
	\draw[thick] (A1) -- (A\m);
	\draw[thick] (B1) -- (B\m);

	\draw[ultra thick, red] (A0) -- (D0);
	\draw[ultra thick, red] (A1) -- (D1);

	\draw[ultra thick, red] (E0) -- (E1);

	\foreach \i in {0,...,\m}
	{
		\draw (A\i) \vx;
		\draw (D\i) \vx;
		\draw (B\i) \vx;
		\draw (E\i) \vx;
	}

\end{tikzpicture}

}

\newcommand{\figsixa}[1]
{
\begin{tikzpicture}[#1]
	\def \r{.7}
	\def \s{.4}

	\def \m{2}
	\def \n{6}
	\def \d{.5*\r}
	\foreach \i in {0,...,\m}
	{
		\defPt{\i*\r}{0}{A\i}
		\ifthenelse{\i = 0}{}{
			\defPt{\i*\r - .5*\r}{\s}{D\i}
		}
		\pgfmathsetmacro{\j}{\n-\i}
		\defPt{\j*\r}{0}{B\i}
		\ifthenelse{\i = 0}{}{
			\defPt{\j*\r + .5*\r}{\s}{E\i}
		}
	}
	\defPt{.5*\r}{-\s}{D0}
	\defPt{\n*\r - .5*\r}{-\s}{E0}

	\foreach \i in {1,...,\m}
	{
		\pgfmathsetmacro{\j}{\i-1}
		\ifthenelse{\i=1}{
			\draw[thick] (A\j) -- (D\i);			
			\draw[thick] (D\i) -- (A\i);
		}{
			\draw[ultra thick, red] (D\i) -- (A\i);
			\draw[ultra thick, green!90!black] (A\j) -- (D\i);						
		}
		\pgfmathsetmacro{\j}{\i-1}
		\ifthenelse{\i=1}{
			\draw[thick] (B\j) -- (E\i);
			\draw[thick] (E\i) -- (B\i);
		}{
			\draw[ultra thick, red] (B\j) -- (E\i);
			\draw[ultra thick, green!90!black] (E\i) -- (B\i);			
		}
	}

	\draw[ultra thick, red] (A0) -- (D1);
	\draw[ultra thick, red] (A1) -- (D0);
	\draw[ultra thick, blue!50] (A0) -- (D0);
	\draw[ultra thick, blue!50] (A1) -- (D1);

	\draw[ultra thick, red] (B0) -- (E0);

	\draw [dash pattern=on \pgflinewidth off 3pt, line width = .7pt] ($(A\m) + (\d, 0)$) -- ($(B\m) + (-\d, 0)$);

	\draw[thick] (A1) -- (A\m);
	\draw[thick] (B1) -- (B\m);

	\draw[thick] (A0) -- (A1);
	\draw[ultra thick, green!90!black] (B0) -- (B1);
	\draw[ultra thick, blue!50] (B0) -- (E0);
	\draw[thick] (B1) -- (E0);

	\draw[ultra thick, red] (B0) -- (E1);

	\foreach \i in {0,...,\m}
	{
		\draw (A\i) \vx;
		\draw (D\i) \vx;
		\draw (B\i) \vx;
		\draw (E\i) \vx;
	}

\end{tikzpicture}

}

\makeatletter
    \def\@fnsymbol#1{\ensuremath{\ifcase#1\or *\or \mathsection\or \ddagger\or
       \dagger\or \mathparagraph\or \|\or **\or \dagger\dagger
       \or \ddagger\ddagger \else\@ctrerr\fi}}
\makeatother

\title{Small rainbow cliques in randomly perturbed dense graphs}

\author{
	Elad Aigner-Horev \thanks{Department of Computer Science, Ariel University, Ariel 40700, Israel. Email: {\tt horev@ariel.ac.il}.} 
	\and 
	Oran Danon \thanks{Department of Computer Science, Ariel University, Ariel 40700, Israel. Email: {\tt oran.danon@msmail.ariel.ac.il}.}
	\and 
	Dan Hefetz \thanks{Department of Computer Science, Ariel University, Ariel 40700, Israel. Email: {\tt danhe@ariel.ac.il}. Research supported by ISF grant 822/18.}
	\and 
	Shoham Letzter \thanks{Department of Mathematics, University College London, Gower Street, London, WC1E 6BT, UK. Email: {\tt s.letzter@ucl.ac.uk}. Research supported by the Royal Society.}
}

\begin{document}

\clearpage\maketitle

\begin{abstract}
	For two graphs $G$ and $H$, write $G \rainbow H$ if $G$ has the property that every \emph{proper} colouring of its edges yields a \emph{rainbow} copy of $H$.
	We study the thresholds for such so-called \emph{anti-Ramsey} properties in randomly perturbed dense graphs, which are unions of the form $G \cup \mathbb{G}(n,p)$, where $G$ is an $n$-vertex graph with edge-density at least $d >0$, and $d$ is independent of $n$. 
	 
	In a companion paper, we proved that the threshold for the property 
	$G \cup \Gnp \rainbow K_\ell$ is $n^{-1/m_2(K_{\ceil{\ell/2}})}$, whenever $\ell 
	\geq 9$. For smaller $\ell$, the thresholds behave more erratically, and for $4 \le \ell \le 7$ they deviate downwards significantly from the aforementioned aesthetic form capturing the thresholds for \emph{large} cliques. 

	In particular, we show that the thresholds for $\ell \in \{4, 5, 7\}$ are $n^{-5/4}$, $n^{-1}$, and $n^{-7/15}$, respectively. For $\ell \in \{6, 8\}$ we determine the threshold up to a $(1 + o(1))$-factor in the exponent: they are $n^{-(2/3 + o(1))}$ and $n^{-(2/5 + o(1))}$, respectively. For $\ell = 3$, the threshold is $n^{-2}$; this follows from a more general result about odd cycles in our companion paper.
\end{abstract}

\section{Introduction}

	A \emph{random perturbation} of a fixed $n$-vertex graph $G$, denoted by $G \cup \mathbb{G}(n,p)$, is a distribution over the supergraphs of $G$. The elements of such a distribution are generated via the addition of randomly sampled edges to $G$. These random edges are taken from the binomial random graph on $n$ vertices with edge-density $p$, denoted $\mathbb{G}(n,p)$. The fixed graph $G$ being \emph{perturbed} or \emph{augmented} in this manner is referred to as the {\em seed} of the {\em perturbation} (or {\em augmentation}) $G \cup \mathbb{G}(n,p)$. Let $\sG_{d,n}$ denote the family of $n$-vertex graphs with edge density at least $d > 0$; the notation $\sG_{d,n} \cup \Gnp$ then suggests itself to mean the collection of distributions arising from the members of $\sG_{d, n}$. 

	The above model of randomly perturbed graphs was introduced by Bohman, Frieze, and Martin~\cite{BFM03}. Since then, two prominent strands of study regarding the distribution of randomly perturbed dense graphs have emerged. The first is the  generalisation of the results of~\cite{BFM03}, regarding the Hamiltonicity of perturbed dense graphs, to the study of spanning structures in said graph distributions. Here, one encounters numerous results such as~\cite{BTW17, BHKM18, BFKM04, BHKMPP18, BMPP18, DRRS18, HZ18, KKS16, KKS17, MM18}. 
	
	The second strand of study, initiated by Krivelevich, Sudakov, and Tetali~\cite{KST}, deals with Ramsey properties of such graph distributions, thus extending the classical results regarding Ramsey properties of random graphs~\cite{LRV92,NS16,RR93,RR94,RR95}. Das and Treglown~\cite{DT19} and Powierski~\cite{Powierski19} significantly extended the body of results set in~\cite{KST} regarding the thresholds of (symmetric and asymmetric) Ramsey properties of the form $\sG_{d,n} \cup \Gnp \to (K_s,K_r)$. Here, $H \to (H_1, H_2)$ is the classical arrow notation used in Ramsey theory to denote that the graph $H$ has the property that every red/blue colouring of its edges admits a red copy of $H_1$ or a blue copy of $H_2$. Additionally, Das and Treglown~\cite{DT19} also study asymmetric Ramsey properties of $\sG_{d,n} \cup \Gnp$ involving cliques and cycles. Das, Morris, and Treglown~\cite{DMT19} extended the results of Kreuter~\cite{K96} pertaining to \emph{vertex Ramsey} properties of random graphs to the perturbed model. In the Ramsey-arithmetic scene, the first author and Person~\cite{AHP} established an (asymptotically) optimal Schur-type theorem for randomly perturbed dense sets of integers. Sudakov and Vondr\'ak~\cite{SV08} studied the non-$2$-colourability of randomly perturbed dense hypergraphs.

	\medskip

	The term {\em anti-Ramsey} is commonly used in order to refer to a body of problems and results concerning the emergence of non-monochromatic configurations in every (sensible) edge-colouring of a given graph. Here, one encounters a large diversity concerning this theme; the reader is referred to the excellent survey~\cite{FMO10} and references therein for more details. 
		
	A subgraph $H \subseteq G$ is said to be {\em rainbow} with respect to an edge-colouring $\psi$, if every two of its edges are assigned different colours under $\psi$, that is, if $|\psi(E(H))| = e(H)$, where $\psi(E(H)) := \{\psi(e) : e \in E(H)\}$ is the set of colours $\psi$ assigns to the edges of $H$; we will often abbreviate $\psi(E(H))$ and write $\psi(H)$ instead. We write $G \rainbow H$, if $G$ has the property that every \emph{proper} colouring of its edges admits a rainbow copy of $H$. 

	A fairly complete overview regarding the emergence of small fixed rainbow configurations in random graphs can be found in the work of Bohman, Frieze, Pikhurko, and Smyth~\cite{BFPS10} and references therein (note that they have considered a wider class of edge-colourings rather than just proper ones). The first to consider the emergence of fixed rainbow configurations in random graphs with respect to proper edge-colourings were R\"odl and Tuza~\cite{RT92}. Subsequently, Kohayakawa, Konstadinidis and Mota~\cite{KKM14, KKM18} launched the systematic study of such rainbow configurations in random graphs with respect to proper edge-colourings. 

	A sequence $\hat{p} := \hat{p}(n)$ is said to form a \emph{threshold} for a property $\P$ if $\mathbb{G}(n, p)$ a.a.s.\ (that is, with probability tending to 1 as $n$ tends to infinity) satisfies $\P$ whenever $p = \omega(\hat{p})$, and $\mathbb{G}(n, p)$ a.a.s.\ does not satisfy $\P$ when $p = o(\hat{p})$.
	The \emph{maximum $2$-density} of a graph $H$ is 
	$$
		m_2(H) := \max \left \{\frac{e(F)-1}{v(F)-2} : F \subseteq H, e(F) \geq 2 \right\}.
	$$ 
	In particular, the maximum $2$-density of a clique is
	$$
		m_2(K_r) := \frac{\binom{r}{2} - 1}{r-2} = \frac{r^2 - r - 2}{2(r-2)} = \frac{r+1}{2}.
	$$
	Kohayakawa, Konstadinidis and Mota \cite{KKM14} proved that for every graph $H$, there exists a constant $C > 0$ such that a.a.s.\ $\mathbb{G}(n,p) \rainbow H$, whenever $p \geq C n^{-1/m_2(H)}$. For $H \cong K_r$ with $r \geq 19$, Nenadov, Person, \v{S}kori\'{c}, and Steger~\cite{NPSS17} proved, amongst other things, that $n^{-1/m_2(H)}$ is the threshold for the property $\mathbb{G}(n,p) \rainbow H$. 
	Kohayakawa, Mota, Parczyk, and Schnitzer~\cite{KMPS18} extended the result of~\cite{NPSS17}, proving that the threshold of the property $\mathbb{G}(n,p) \rainbow K_r$ remains $n^{-1/m_2(K_r)}$ for every $r \geq 5$.  

	For $K_4$ the situation is different. The threshold for the property $\mathbb{G}(n,p) \rainbow K_4$ is $n^{-7/15} = o \left(n^{-1/m_2(K_4)} \right)$, as proved by Kohayakawa, Mota, Parczyk, and Schnitzer~\cite{KMPS18}. More generally, Kohayakawa, Konstadinidis and Mota~\cite{KKM18} proved that there are infinitely many graphs $H$ for which the threshold for the property $\mathbb{G}(n,p) \rainbow H$ is significantly smaller than $n^{-1/m_2(H)}$. 

	Note that the threshold for the property $\mathbb{G}(n,p) \rainbow K_3$ coincides with the 
	threshold for the emergence of $K_3$ in $\mathbb{G}(n,p)$, which is $n^{-1}$, as 
	every properly-coloured triangle is rainbow.

	\bigskip

	For a real $d > 0$, we say that $\sG_{d,n} \cup \mathbb{G}(n,p)$ a.a.s.\ satisfies a graph property $\P$, if 
	$$
		\lim_{n \to \infty} \Pro[G_n \cup \mathbb{G}(n,p) \in \P] = 1
	$$
	holds for \emph{every} sequence $\{G_n\}_{n \in \mathbb{N}}$ satisfying $G_n \in \sG_{d,n}$ for every $n \in \mathbb{N}$. We say that $\sG_{d,n} \cup \mathbb{G}(n,p)$ a.a.s.\ does not satisfy $\P$, if 
	$$
		\lim_{n \to \infty} \Pro[G_n \cup \mathbb{G}(n,p) \in \P] = 0
	$$
	holds for at \emph{least} one sequence $\{G_n\}_{n \in \mathbb{N}}$ satisfying $G_n \in \sG_{d,n}$ for every $n \in \mathbb{N}$. 
	A sequence $\widehat{p}:=\widehat{p}(n)$ is said to form a {\em threshold} for the property $\P$ in the perturbed model, if $\sG_{d,n} \cup \mathbb{G}(n,p)$ a.a.s.\ satisfies $\P$ whenever $p = \omega(\widehat{p})$, and if 
	$\sG_{d,n} \cup \mathbb{G}(n,p)$ a.a.s.\ does not satisfy $\P$  whenever $p = o(\widehat{p})$.
	
	Throughout, we suppress this sequence-based terminology and write more concisely that $\sG_{d,n} \cup \mathbb{G}(n,p)$ a.a.s.\ satisfies (or does not) a certain property. In particular, for a fixed graph $H$, we say that $\sG_{d,n} \cup \mathbb{G}(n,p) \rainbow H$ holds a.a.s.\ if the aforementioned anti-Ramsey property is upheld a.a.s.\  by \emph{every} sequence of graphs in $\sG_{d,n}$. We say that $\sG_{d,n} \cup \mathbb{G}(n,p) \notrainbow H$ holds a.a.s.\ if there \emph{exists} a sequence of graphs in $\sG_{d,n}$ for which the property fails asymptotically almost surely. 

	\medskip		

	In this paper we are interested in determining the threshold for the property $\sG_{d, n} \cup \Gnp \rainbow K_{\ell}$, for fixed $\ell \ge 3$. Our results in this paper are complemented by our companion paper~\cite{companion}. Here is an \emph{abridged} version of the main result of the companion paper.

	\begin{theorem}\label{thm:main:companion} {\em~\cite[Proposition~5.1]{companion}}
		Let a real number $0 < d < 1$ and an integer $\ell \geq 5$ be given. 
		Then, the property $\sG_{d,n} \cup \mathbb{G}(n,p) \rainbow K_{\ell}$ holds a.a.s.\ whenever $p := p(n) = \omega\big(n^{-1/m_2(K_{\ceil{\ell/2}})}\big)$. 
	\end{theorem}


	Theorem~\ref{thm:main:companion} in conjunction with the aforementioned results of~\cite{KMPS18,NPSS17} assert that $n^{-1/m_2(K_{\ceil{\ell/2}})}$ is the threshold for the property $\sG_{d,n} \cup \mathbb{G}(n,p) \rainbow K_{\ell}$, whenever $\ell \geq 9$ and $0 < d \leq 1/2$. 
	Indeed, given a sufficiently large integer $n$, take $G$ to be a bipartite graph on $n$ vertices with edge-density at least $d$ (such a graph exists by the assumption $d \leq 1/2$) and denote its bipartition by $\{X, Y\}$. Since $\ell \geq 9$, it follows by~\cite{KMPS18} that, if $p = o(n^{-1/m_2(K_{\ceil{\ell/2}})})$, then a.a.s.\ there is a proper edge-colouring $\psi$ of $\mathbb{G}(n, p)$ without any rainbow copies of $K_{\ceil{\ell/2}}$. Consider the edge-colouring of $G \cup \mathbb{G}(n, p)$ obtained by colouring the random edges according to $\psi$ and colouring each of the remaining edges by a unique new colour. The resulting colouring is a proper colouring with no rainbow copies of $K_{\ell}$. This shows that the threshold $\hat{p}$ for the aforementioned property satisfies $\hat{p} = \Omega(n^{-1/m_2(K_{\ceil{\ell/2}})})$, as claimed.
	It follows from another result of~\cite{companion}, regarding odd cycles, that $n^{-2}$ is the threshold for the property $\sG_{d,n} \cup \mathbb{G}(n,p) \rainbow K_3$.  
	
	Since our lower bounds only hold when $d \leq 1/2$ (as they rely on the existence of a bipartite graph with edge-density $d$), one may wonder what happens for larger values of $d$. This issue was considered for the containment problem in~\cite{BFKM04}. In that paper, the range $(0,1]$ of possible values of $d$ was divided into segments and the appropriate threshold for each such segment was determined. Indeed, given $r \geq 3$, if $d > \frac{r-2}{r-1}$, then by Tur\'an's Theorem~\cite{Turan} any graph in $\sG_{d,n}$ admits a copy of $K_r$ and so no random perturbation is required. For smaller values of $d$, it follows by Tur\'an's Theorem that the regularity graph of any $H \in \sG_{d,n}$ admits a copy of $K_t$ for some $2 \leq t < r$. The subgraph of $H$ corresponding to this copy of $K_t$ can then be augmented by the appropriate number of random edges to yield a copy of $K_r$. A \emph{rainbow} variant of Tur\'an's Theorem due to Keevash, Mubayi, Sudakov and Verstra\"ete \cite{KMSV07} implies that no random perturbation is needed in our setting either whenever $d > \frac{r-2}{r-1}$. Moreover, it is plausible that if $H \in \sG_{d,n}$ admits a copy of $K_t$ for some $2 < t < r$, then this copy may be used to decrease the number of random edges needed to ensure a rainbow copy of $K_r$. This seems to complicate our arguments (which are already quite long and involved) and so we have chosen not to pursue this endeavour in the present paper.


	\subsection{Our results}\label{sec:our-results}

		Theorem~\ref{thm:main:companion} does not apply to $\ell = 4$. Moreover, while it does provide an upper bound on the threshold for the property $\sG_{d,n} \cup \mathbb{G}(n,p) \rainbow K_{\ell}$ for every $5 \leq \ell \leq 8$, a matching lower bound is not known to hold. For $4 \leq  \ell \leq 7$, it turns out that $n^{-1/m_2(K_{\ceil{\ell/2}})}$ is not the threshold of the corresponding property; indeed the threshold deviates downwards quite significantly from this function. Our first main result determines the threshold for the associated properties when $\ell \in \{4,5,7\}$. 

		\begin{theorem}\label{thm:main:457}
			Let $0 < d \leq 1/2$ be given. 
			\begin{enumerate}
				\item The threshold for the property $\sG_{d,n} \cup \mathbb{G}(n,p) \rainbow 
					K_4$ is $n^{-5/4}$.
							
				\item The threshold for the property $\sG_{d,n} \cup \mathbb{G}(n,p) \rainbow 
							K_5$ is $n^{-1}$. 
							
				\item The threshold for the property $\sG_{d,n} \cup \mathbb{G}(n,p) \rainbow 
							K_7$ is $n^{-7/15}$. 
					
			\end{enumerate}
		\end{theorem}

		For $K_6$, we prove the following. 

		\begin{theorem}\label{thm:main:6}
			Let $0 < d \leq 1/2$ be given. 
			\begin{enumerate}
				\item The property $\,\sG_{d,n} \cup \mathbb{G}(n,p) \rainbow K_6$ holds a.a.s.\  
				whenever $p = \omega(n^{-2/3})$. 
				
				\item For every $\eps >0$, the property $\,\sG_{d,n} \cup \mathbb{G}(n,p) \notrainbow K_6$ 
				holds a.a.s., whenever $p := p(n) = n^{-(2/3 + \eps)}$. 
			\end{enumerate}
		\end{theorem}

		For $K_8$, Theorem~\ref{thm:main:companion} asserts that $\sG_{d,n} \cup \mathbb{G}(n,p) \rainbow K_8$ holds a.a.s.\  whenever $p := p(n) = \omega(n^{-2/5}) = \omega(n^{-1/m_2(K_4)})$. We prove the following almost matching lower bound on the threshold of the property $\sG_{d, n} \cup \Gnp \rainbow K_8$. We observe that the simpler argument, presented after the statement of Theorem~\ref{thm:main:companion}, provides a weaker lower bound on this threshold, namely $n^{-7/15} = o(n^{-2/5})$, which is the threshold of the property $\Gnp \rainbow K_4$.

		\begin{theorem}\label{thm:main:8}
			For every $0 < d \leq 1/2$ and $\eps > 0$, the property $\sG_{d,n} \cup \mathbb{G}(n,p) \notrainbow K_8$ holds a.a.s., whenever $p:= p(n) = n^{-(2/5+\eps)}$. 
		\end{theorem}

		The proof of Theorem~\ref{thm:main:companion} in~\cite{companion} relies heavily on the so-called K{\L}R-theorem~\cite[Theorem~1.6(i)]{CGSS14}; 
		the proofs of all the results stated above, employ entirely different approaches. Indeed, more refinement and control are required in order to handle \emph{small} cliques. 

		We summarise our results regarding the threshold of the property $\sG_{d, n} \cup \Gnp \rainbow K_{\ell}$ in Table~\ref{table}. Note that, as indicated above, our lower bounds apply only when $d \in (0, 1/2]$, whereas the upper bounds apply for every $d \in (0, 1]$. 

		\renewcommand{\arraystretch}{1.45}
		\begin{table}[h!]
			\centering
			\begin{tabular}{|c|c|c|}
				\hline
				$\ell$ & lower bound for $d \in (0, 1/2]$ & upper bound for $d \in (0, 1]$ \\
				\hline
				\hline
				$3$ &  \multicolumn{2}{c|}{\cellcolor{pink} $\Theta(n^{-2})$} \\
				\hline
				$4$ & \multicolumn{2}{c|}{\cellcolor{blue!25} $\Theta(n^{-5/4})$} \\
				\hline
				$5$ & \multicolumn{2}{c|}{\cellcolor{blue!25} $\Theta(n^{-1})$} \\
				\hline
				$6$ & \cellcolor{orange!50} $\Omega(n^{-(2/3 + \eps)})$ for any fixed $\eps > 0$ & \cellcolor{orange!50} $O(n^{-2/3})$ \\
				\hline
				$7$ & \multicolumn{2}{c|}{\cellcolor{blue!25} $\Theta(n^{-7/15})$} \\
				\hline
				$8$ & \cellcolor{green!50} $\Omega(n^{-(2/5 + \eps)})$ for any fixed $\eps > 0$ & \cellcolor{pink} $O(n^{-2/5})$ \\
				\hline
				$\ge 9$ & \multicolumn{2}{c|}{\cellcolor{pink} $\Theta\left(n^{-1/m_2(K_{\ceil{\ell/2}})}\right)$} \\ 
				\hline
			\end{tabular}
			\caption{A table summarising the results regarding the threshold for the property $\sG_{d, n} \cup \Gnp \rainbow K_{\ell}$. The colours refer to where the result is proved: pink refers to our companion paper \cite{companion}; blue refers to Theorem~\ref{thm:main:457}; orange refers to Theorem~\ref{thm:main:6}; and green refers to Theorem~\ref{thm:main:8}.}

			\label{table}
		\end{table}

		The rest of the paper is organized as follows. In Section~\ref{sec:prelims} we mention some preliminaries and useful observations. We consider the thresholds of the properties $\sG_{d, n} \cup \Gnp \rainbow K_{\ell}$ for $\ell \in \{4, 5, 6, 7, 8\}$ in Sections~\ref{sec:K4}, \ref{sec:K5}, \ref{sec:K6}, \ref{sec:K7}, and \ref{sec:K8}, respectively.

\section{Preliminaries} \label{sec:prelims}
	 
	For a graph $G$ and a set $X$ of vertices, the \emph{common neighbourhood} of $X$ in $G$, denoted $N_G(X)$, is the intersection of the neighbourhoods of vertices in $X$, namely $N_G(X) := \bigcap_{x \in X} N_G(x)$. Given two disjoint sets $X$, $Y$ of a vertices in a graph $G$, define $E_G(X, Y)$ to be the set of edges in $G$ with one end in $X$ and the other in $Y$, and let $e_G(X, Y) = |E_G(X, Y)|$.

	We write $G \sim \Gnp$ to mean that $G$ is a random graph sampled according to the distribution of $\Gnp$. Similarly, we write $\Gamma \sim \sG_{d, n} \cup \Gnp$ to mean that $\Gamma$ is obtained by taking the union of some graph $G_1$ in $\sG_{d, n}$ and a random graph $G_2 \sim \Gnp$ (where the graphs $G_1$ and $G_2$ have the same vertex set).

	Given a sequence $f := f(n)$ and constants $\eps_1, \ldots, \eps_k > 0$ independent of $n$, we write $\Omega_{\eps_1,\ldots,\eps_k}(f)$, $\Theta_{\eps_1,\ldots,\eps_k}(f)$, and $O_{\eps_1,\ldots,\eps_k}(f)$ to mean that the constants which are implicit in the asymptotic notation depend on $\eps_1, \ldots, \eps_k$. We will occasionally replace these constants with fixed graphs, writing $O_L(f)$ to indicate that the implicit constants in the asymptotic notation depend on the graph $L$. In addition, given two constants $\mu > 0$ and $\nu > 0$ we write $\mu \ll \nu$ to mean that, while $\mu$ and $\nu$ are fixed, they can be chosen so that $\mu$ is arbitrarily smaller than $\nu$. 

	Throughout, in the proofs of the $1$-statements we make repeated (standard) appeals to the so-called \emph{dense} regularity lemma~\cite{Szemeredi78} (see also~\cite{KS96}). For a bipartite graph $G := (U \discup W, E)$ and two sets $U' \subseteq U$ and $W' \subseteq W$, write $d_G(U',W') := \frac{e_G(U',W')}{|U'||W'|}$ for the edge-density of the induced subgraph $G[U',W']$. The graph $G$ is called $\eps$-{\em regular} if 
	$$
	|d_G(U',W') - d_G(U,W)| < \eps  
	$$
	holds whenever $U' \subseteq U$ and $W' \subseteq W$ satisfy $|U'| \geq \eps |U|$ and $|W'| \geq \eps |W|$. 

	\subsection{Sparse bipartite graphs}\label{sec:sparse-bip}

		For two vertex disjoint graphs $L$ and $R$, let $K_{L,R}$ denote the \emph{join of $L$ and $R$}, namely the graph $(V(L) \discup V(R),F)$, where 
		$$
			F:= E(L) \discup E(R) \discup \{\ell r: \ell \in V(L), r \in V(R)\}.
		$$
		In the special case that $e(R) = 0$, we write $K_{L, v(R)}$ instead; further still, if in addition $L$ is complete, then we write $\widehat{K}_{v(L), v(R)}$. 
		We denote by $\tilde L$ and $\tilde R$ the realisations of $L$ and $R$ in $K_{L,R}$. 

		Let $G$ be a graph and let $K\subseteq G$ be a subgraph of $G$. Let $\psi$ be a proper edge-colouring of $G$. If $K$ appears rainbow under $\psi$, then $K$ is said to be $\psi$-{\em rainbow}. A vertex $x$ 
		found in the common neighbourhood $N_G(V(K))$ is said to be of {\em interest to $K$ with respect to $\psi$} if 
		$$
			\psi(K) \cap \{\psi(xk): k \in V(K)\} = \emptyset.
		$$
		If, in addition, $K$ is $\psi$-rainbow, then the above definition stipulates that $G[V(K) \cup \{x\}]$ is $\psi$-rainbow (though, perhaps unintuitively, we make use of the more general definition). A set $X \subseteq V(G)$ whose members are \emph{all} of interest to $K$ with respect to $\psi$, is said to be {\em compatible with $K$ with respect to $\psi$} provided that the sets $\{\psi(xk): k \in V(K)\}$ are pairwise disjoint for $x \in X$.
		If, in addition, $K$ is known to be $\psi$-rainbow, then the latter definition stipulates that $G[V(K)] \cup G[V(K),X]$ is $\psi$-rainbow. 

		\begin{observation}\label{obs:interest}
			Let $L$ be a fixed graph and let $n$ be sufficiently large. Every proper edge-colouring $\psi$ of $K_{L,n}$ admits a subset $I_\psi := I_\psi(\tilde L)  \subseteq V(\tilde{R})$, satisfying $|I_\psi| = n - O_L(1)$, such that all of its members are of interest to $\tilde L$ with respect to $\psi$. 
		\end{observation}

		\begin{proof}
			Being proper, the colour classes of $\psi$ define (pairwise edge-disjoint) matchings in $K_{L,n}$.
			Hence, for each colour in $\psi(\tilde{L})$, there are at most $v(\tilde L) - 2$ vertices outside of $\tilde L$ that send an edge of this colour to $\tilde L$. It follows that there are at most $e(L)(v(L) - 2)$ vertices outside of $\tilde L$ that send an edge of a colour present in $\psi(\tilde L)$ to $\tilde L$. The claim follows.
		\end{proof}

		\begin{observation}\label{obs:compatible} 
			Let a graph $L$ be fixed and let $n$ be sufficiently large. Every proper edge-colouring $\psi$ of $H:= K_{L,n}$ admits a set $C_\psi:= C_\psi(\tilde L) \subseteq V(\tilde{R})$, satisfying $|C_\psi| = \Omega_L(n)$, that is compatible with $\tilde L$ with respect to $\psi$. 
		\end{observation}

		\begin{proof}
			Fix a proper edge-colouring $\psi$ of $H$. Let $I_\psi = I_\psi(\tilde L)$ be the set whose existence is ensured by Observation~\ref{obs:interest} and let $\{u_1,\ldots,u_t\}$ be an arbitrary ordering of its elements; note that $t = \Omega_L(n)$ holds by Observation~\ref{obs:interest}. The set $C_\psi$ is constructed recursively as follows. Initially, we set $C_\psi = \{u_1\}$ and proceed to iterate over $I_\psi$ according to the ordering of its elements fixed above, making a decision for each member considered whether or not to include it in the set $C_\psi$.   

			Suppose that for some $1 \leq j \leq t-1$, the decision on whether or not to include $u_i$ in $C_\psi$ has been made for every $1 \leq i \leq j$, and that the current set $C_\psi$ is compatible with ${\tilde L}$ with respect to $\psi$; this trivially holds for $j=1$. Add $u_{j+1}$ to $C_\psi$ if and only if $C_\psi \cup \{u_{j+1}\}$ is compatible with $\tilde L$ with respect to $\psi$. Since $\psi$ is proper, each vertex added to $C_\psi$ \emph{disqualifies} at most $v(L) (v(L) - 1)$ vertices in $I_\psi(\tilde L)$ from being added in subsequent rounds, as each of the $v(L)$ colours appearing on the edges incident with that vertex forms a matching of size at most $v(L)$. Hence, at least $n/O_L(1)$ vertex-additions are performed throughout the above process and the claim follows. 
		\end{proof}


		\subsubsection*{Overview of the $1$-statement proofs}
			The proofs of the $1$-statements for $K_4$, $K_5$, $K_6$ and $K_7$ follow a similar pattern (with a simpler version for $K_4$). We thus give a short overview of these proofs here. Given $t \in [4, 7]$, a constant $d \in (0, 1]$, an appropriate $p$, a sufficiently large $n$, and a graph $G \in \sG_{d, n}$, we apply the regularity lemma to find an $\eps$-regular bipartite subgraph of $G$ with density at least $d'$, denoted $G[U, W]$, such that $|U|, |W| = \Omega(n)$, where $\eps, d'$ are small constants. We abuse notation slightly by assuming that $G = G[U, W]$, and take $G_1 = (\Gnp)[W]$ and $G_2 = (\Gnp)[U]$. To prove the $1$-statement it suffices to show that a.a.s. $\Gamma := G \cup G_1 \cup G_2 \rainbow K_t$. 

			In each case we judiciously pick fixed graphs $H$ and $F$ such that a.a.s. $G_1$ contains a copy of $H$ and every linear subset of vertices in $G_2$ spans a copy of $F$. To prove that these properties hold a.a.s.\ we use standard tools which are stated in Appendix~\ref{sec:Janson}. It follows quite easily from the regularity of $G$ and Observation~\ref{obs:compatible} that for $G_1$ and $G_2$ that satisfy the aforementioned properties and for any proper colouring $\psi$ of $\Gamma$, there is a copy of $K := K_{H, F}$ in $\Gamma$ such that $V(\tilde{F})$ is compatible with $\tilde{H}$ with respect to $\psi$, where $\tilde{H}$ and $\tilde{F}$ are the natural embeddings of $H$ and $F$ in $K$. To complete the proof of the $1$-statement it suffices to show that such a copy of $K$ contains a rainbow $K_t$.

\section{Rainbow copies of $K_4$}\label{sec:K4}

	In this section we prove the first part of Theorem~\ref{thm:main:457} asserting that the threshold for the property $\sG_{d,n} \cup \mathbb{G}(n,p) \rainbow K_4$ is $n^{-5/4}$. 

	\subsection{$1$-statement}

		Let $d \in (0, 1]$ be fixed, let $p := p(n) = \omega(n^{-5/4})$, and let $n$ be sufficiently large. Let $G \in \sG_{d, n}$ and let $G_1 \sim \Gnp$. We will show that a.a.s. $G \cup G_1 \rainbow K_4$, thus proving the $1$-statement of the first part of Theorem~\ref{thm:main:457}.

		\begin{claim} \label{clm:one-K4}
			Asymptotically almost surely $G \cup G_1$ contains a copy of $H := K_{K_{1,3}, K_{1, 4}}$.
		\end{claim}

		\begin{proof}
			We will use a result of Krivelevich, Sudakov, and Tetali~\cite{KST} (in fact, we shall only need its $1$-statement). To state their result, we need two definitions.
			The \emph{maximum density} of a graph $J$ is the following quantity 
			$$
			m_1(J) := \max \{e(J')/v(J'): J' \subseteq J, v(J') > 0\}.
			$$
			The {\em maximum bipartition density} of $J$ is given by 
			$$
			m_{(2)}(J) := \min_{V(J) = V_1 \discup V_2} \max \left\{m_1(J[V_1]), m_1(J[V_2]) \right\}.
			$$

			\begin{theorem}\label{thm:KST} {\em (\cite[Theorem~2.1]{KST} -- abridged)} 
				For every real $d \in (0, 1]$, fixed graph $J$, and $G \in \sG_{d, n}$, the perturbed graph $G \cup \mathbb{G}(n,p)$ a.a.s.\  contains a copy of $J$, whenever $p := p(n) = \omega(n^{-1/m_{(2)}(J)})$.
			\end{theorem}
			Observe that
			$$
				m_{(2)}(H) = \max \left\{m_1(K_{1,3}), m_1(K_{1,4}) \right\} = \max \left\{3/4, 4/5 \right\} = 4/5.
			$$
			Claim~\ref{clm:one-K4} thus follows by Theorem~\ref{thm:KST} and the assumption that $p = \omega(n^{-5/4})$.
		\end{proof}

		We use the following observation.

		\begin{observation}\label{obs:reduce-to-containment}
			$H \rainbow K_4$. 
		\end{observation}

		\begin{proof}
			Fix an arbitrary proper colouring $\psi$ of the edges of $K_{L,R}$, where $L \cong K_{1,3}$ and $R \cong K_{1,4}$. Let $e \in E(R)$ be an edge for which $\psi(e) \notin \{\psi(e') : e' \in E(L)\}$ holds. It is now straightforward to verify that the graph $K_{L,e}$ contains a copy of $K_4$ which is rainbow under $\psi$.   
		\end{proof}

		The $1$-statement for $K_4$ is an immediate corollary of Claim~\ref{clm:one-K4} and Observation~\ref{obs:reduce-to-containment}.

	\subsection{$0$-statement}

		Let $G := (U \discup W,E) \cong K_{\lfloor n/2 \rfloor, \lceil n/2 \rceil}$ and let $p := p(n) = o \left(n^{-5/4} \right)$. We prove that a.a.s.\  $G \cup \Gnp \notrainbow K_4$ holds, by describing a proper colouring of the edges of $G \cup \Gnp$ admitting no rainbow $K_4$. With $p$ significantly below the threshold for the emergence of $K_3$ in $\mathbb{G}(n,p)$ (see, e.g.,~\cite[Theorem~3.4]{JLR}), the random perturbation $\Gnp$ itself is a.a.s.\  triangle-free. Consequently, $G \cup \mathbb{G}(n,p)$ a.a.s.\  has the property that all its copies of $K_4$ are comprised of a copy of $C_4$, present in $G$, and two additional edges brought on by the perturbation $\mathbb{G}(n,p)$ such that one is spanned by $U$ and the other by $W$. 

		With $p$ being below the threshold for the emergence of triangles, $4$-cycles, and any connected graph on five vertices in $\mathbb{G}(n,p)$ (see, e.g.,~\cite[Theorem~3.4]{JLR}), it follows that a.a.s.\  the edges of the perturbation are captured through a collection of vertex-disjoint copies of $K_2$, $P_3$, $K_{1,3}$, and $P_4$, where $P_i$ is the path on $i$ vertices. Let $G' \sim \mathbb{G}(n,p)$ having this component structure be fixed and let $\Gamma = G \cup G'$. Then, every copy of $K_4$ in $\Gamma$ is found within some copy of $K_{L,R}$, with $L,R \in \{K_2, P_3, K_{1,3}, P_4\}$ and such that $V(L) \subseteq U$ and $V(R) \subseteq W$. 

		Let $L_1, \ldots, L_s$ and $R_1, \ldots, R_t$ be arbitrary enumerations of the connected components of $\Gamma[U]$ and  $\Gamma[W]$, respectively; so $L_i$ and $R_j$ are copies of one of $K_1, K_2, P_3, K_{1,3}, P_4$ for every $i \in [s]$ and $j \in [t]$. In what follows, we define a colouring of $\Gamma$ in which all of the aforementioned components appearing in $\Gamma[U]$ and $\Gamma[W]$ are coloured using the colours $1,2$, and $3$. For each pair $(i, j)$ we assign a set $A_{i,j}$ of $|V(L_i)| \cdot |V(R_j)|$ colours to be used on the edges from $R_i$ to $L_j$ (when colouring these edges, we may repeat colours, thus not using all of the colours in $A_{i,j}$), such that the sets $A_{i,j}$ are pairwise disjoint and do not intersect $\{1, 2, 3\}$. We obtain a proper edge-colouring of $\Gamma$ as follows. 
		
		\begin{enumerate}[leftmargin = *, label = \bf (A\arabic*)]
			\item \label{itm:colour-inside-edges} 
				Colour the edges of each connected component $C \in \{L_1, \ldots, L_s, R_1, \ldots, R_t\}$ as follows.
			\begin{enumerate}
				\item
					If $C$ is a single vertex, there is nothing to colour.
				\item 
					If $C \cong K_2$, colour its edge using the colour $1$. 

				\item 
					If $C \cong P_3$, colour it properly using the colours $1,2$.

				\item 
					If $C \cong K_{1,3}$, colour it properly using the colours $1,2,3$.

				\item 
					If $C \cong P_4$, colour it properly using the colours $1,2,3$ such that all three colours are used and the colour $2$ is used for the middle edge. 
			\end{enumerate}
			
			\item \label{itm:colour-cross-edges}
				Given any $1 \leq i \leq s$ and $1 \leq j \leq t$, colour the edges of $G$ connecting $L_i$ and $R_j$ properly, using colours from the set $A_{ij}$, such that the corresponding copy of $K_{L_i, R_j}$ admits no rainbow copy of $K_4$. (The validity of this step is verified below.) 
		\end{enumerate}

		It is evident that the proposed colouring, if it exists, is proper and admits no rainbow copy of $K_4$. Proving that the desired colouring exists, can be done by a fairly straightforward yet somewhat tedious case analysis. It suffices to describe a colouring $\psi_{ij} : V(L_i) \times V(R_j) \to A_{ij}$ for every $1 \leq i \leq s$ and $1 \leq j \leq t$, such that the following holds. Let $\varphi_{ij}$ be the colouring of the edges of $K_{L_i, R_j}$ under which the edges of $V(L_i) \times V(R_j)$ are coloured as in $\psi_{ij}$ and the edges of $E(L_i) \cup E(R_j)$ are coloured as in Item \ref{itm:colour-inside-edges} above. Then, for every $1 \leq i \leq s$ and $1 \leq j \leq t$, the colouring $\varphi_{ij}$ is proper and no copy of $K_4$ in $K_{L_i, R_j}$ is rainbow under $\varphi_{ij}$. 

		It thus suffices to describe such a colouring of $V(L) \times V(R)$ for any $L, R \in \{K_1, K_2, P_3, K_{1, 3}, P_4\}$. Observe that $K_{1,3}$ and $P_4$ contain $K_1$, $K_2$ and $P_3$, where the edges of all five graphs are coloured per articles \ref{itm:colour-inside-edges} and \ref{itm:colour-cross-edges} specified above. Therefore, the desired colouring for $K_{L, R}$, where $L, R \in \{K_{1,3}, P_4\}$, would yield the desired colouring for $K_{L', R'}$ for every $L', R' \in \{K_1, K_2, P_3, K_{1, 3}, P_4\}$. Hence, up to symmetry, we are left with only three cases to consider. In each case we describe an appropriate colouring; verifying that it is proper and yields no rainbow $K_4$ is straightforward and the details are thus omitted.

		\begin{enumerate}
			\item \emph{$L \cong R \cong K_{1,3}$}. 
				Let $\{y, x_1, x_2, x_3\}$ and $\{y', x_1', x_2', x_3'\}$ be the vertices of $L$ and $R$, respectively, with $y$ and $y'$ being the vertices of degree $3$, and $yx_i$ and $y'x_i'$ being the edges of colour $i$. Define the colouring $\psi$ of $V(L) \times V(R)$ as follows. 
			
			\vspace{1ex}
				
			\begin{minipage}[b]{.4\textwidth}

				$\psi (\{x_1 x'_2, y y', x_2 x'_3, x_3 x'_1\}) = \{4\}$ 

				\vspace{.1cm}

				$\psi (\{y x'_2, x_3 y'\}) = \{5\}$

				\vspace{.1cm}

				$\psi (\{x_1 y', y x'_3\}) = \{6\}$ 

				\vspace{.1cm}

				$\psi (\{x_2 y', y x'_1\}) = \{7\}$.

				\vspace{.75cm}
			\end{minipage}
			\hfill
			\begin{minipage}[t]{.5\textwidth}
				\KvsK{scale=0.6}
			\end{minipage}

			\medskip
			To complete the definition of $\psi$, colour each remaining edge using a new unique colour.  
			
			\item \emph{$L \cong K_{1,3}$ and $R \cong P_4$}. 
				Let $\{y, x_1, x_2, x_3\}$ be the vertices of $L$, with $y$ being the vertex of degree $3$ and $yx_i$ having colour $i$ for $i \in [3]$; and let $\{x'_1, x'_2, x'_3, x'_4\}$ be the vertices of $R$ giving rise to the path $x'_1 x'_2 x'_3 x'_4$ with $x_1' x_2'$ coloured $1$.
				Define the colouring $\psi$ of $V(L) \times V(R)$ as follows.
			
				\vspace{1ex}
				
				\begin{minipage}[b]{.3\textwidth}
					$\psi (\{y x'_1, x_2 x'_2\}) = \{4\}$

					\vspace{.1cm}

					$\psi (\{y x'_4, x_2 x'_3\}) = \{5\}$

					\vspace{.1cm}

					$\psi (\{x_1 x'_3, y x'_2, x_3 x'_1\}) = \{6\}$

					\vspace{.1cm}

					$\psi (\{x_1 x'_4, y x'_3, x_3 x'_2\}) = \{7\}$

					\vspace{1.8cm}
				\end{minipage}
				\hfill
				\begin{minipage}[t]{.55\textwidth}
					\hspace{2em}\KvsPFour{scale=0.6}
				\end{minipage}

				\medskip
				To complete the definition of $\psi$, colour each remaining edge using a new unique colour.

			\item \emph{$L \cong R \cong P_4$}. 
				Let $\{x_1, x_2, x_3, x_4\}$ and $\{x_1', x_2', x_3', x_4'\}$ be the vertices of $L$ and $R$, respectively, giving rise to the paths $x_1 x_2 x_3 x_4$ and $x_1'x_2'x_3'x_4'$, with $x_1x_2$ and $x_1'x_2'$ coloured $1$. Define the colouring $\psi$ of $V(L) \times V(R)$ as follows.
			
				\vspace{1ex}
				
				\begin{minipage}[b]{.35\textwidth}
					$\psi (\{x_1 x'_3, x_2 x'_4, x_3 x'_1, x_4 x'_2\}) = \{4\}$ 

					\vspace{.1cm}

					$\psi (\{x_1 x'_2, x_2 x'_3, x_3 x'_4\}) = \{5\}$ 

					\vspace{.1cm}

					$\psi (\{x_2 x'_1, x_3 x'_2, x_4 x'_3\}) = \{6\}$.

					\vspace{1.8cm}
					\end{minipage}
					\hfill
					\begin{minipage}[t]{.55\textwidth}
					\hspace{7em}\PFourVsPFour{scale=0.5}
				\end{minipage}

				\medskip
				To complete the definition of $\psi$, colour each remaining edge using a new unique colour.
		\end{enumerate}
\section{Rainbow copies of $K_5$}\label{sec:K5}

	In this section we prove the second part of Theorem~\ref{thm:main:457} asserting that the threshold for the property $\sG_{d,n} \cup \mathbb{G}(n,p) \rainbow K_5$ is $n^{-1}$. To see the $0$-statement, fix some $d \leq 1/2$ and let $G$ be a bipartite graph on $n$ vertices with density $d$, and let $p = o(1/n)$. Since $G$ is bipartite, any copy of $K_5$ in $\Gamma \sim G \cup \Gnp$ must contain some triangle of $\Gnp$. However, $\Gnp$ is a.a.s.\  triangle-free whenever $p = o(1/n)$. In particular, a.a.s.\  no edge-colouring of $\Gamma$ can yield a rainbow $K_5$.  

	\medskip

	Proceeding to the $1$-statement, let $d \in (0, 1]$ be fixed, let $p := p(n) = \omega(1/n)$, and let $n$ be a sufficiently large integer. Suppose that $G \in \sG_{d,n}$. 

	By a standard application of the (dense) regularity lemma~\cite{Szemeredi78} (see also~\cite{KS96}), we may assume that $G$ is an $\eps$-regular bipartite graph of edge-density $d'$ with bipartition $\{U, W\}$ satisfying $|U| = |W| = m = \Theta_{d',\eps}(n)$, where $\eps$ and $d'$ are sufficiently small positive constants. Let $G_1 \sim (\Gnp)[W]$ and $G_2 \sim (\Gnp)[U]$; observe that the distribution of $G_1$ and $G_2$ is the same as that of $\mathbb{G}(m, p)$. We shall consider the graph $\Gamma := G \cup G_1 \cup G_2$, which is a subgraph of $G \cup \Gnp$. To complete the proof of the $1$-statement for $K_5$, it suffices to show that a.a.s.\ $\Gamma \rainbow K_5$.

	Let $\bar{K}_{3,5} = K_{K_{3}, K_{1,4}}$ and let $\widehat{K}_{3,5}$ be the join of a triangle and an independent set of size $5$.
	The following claim captures the principal property we require $\Gamma$ to satisfy. 

	\begin{claim} \label{cl::B1}
		Asymptotically almost surely any proper edge-colouring $\psi$ of $\Gamma$ admits a copy of $\bar K_{3,5}$ whose copy of $\widehat{K}_{3,5}$ obtained by removing the edges in $\tilde{R}$ is rainbow under $\psi$.
	\end{claim}

	Prior to proving Claim~\ref{cl::B1}, we use it to derive the $1$-statement for $K_5$. Fix $G_1, G_2$ satisfying the property described in Claim~\ref{cl::B1}, and fix a proper colouring $\psi$ of the edges of $\Gamma$. By Claim~\ref{cl::B1}, there exists a copy $K$ of $\bar K_{3,5}$ in $\Gamma$ whose copy $K'$ of $\widehat{K}_{3,5}$, obtained by removing the edges of $\tilde{R}$, is rainbow under $\psi$. Let $V(\tilde{L}) = \{x_1, x_2, x_3\}$ and $V(\tilde{R}) = \{y, z_1, z_2, z_3, z_4\}$, where $y$ is the central vertex of the star $\tilde R$. Since $\psi$ is proper, there exists $t \in [4]$ such that $\psi(y z_t) \notin \{\psi(x_1 x_2), \psi(x_1 x_3), \psi(x_2 x_3)\}$. Using yet again the fact that $\psi$ is proper, it follows that $\psi(y z_t) \notin \psi(\{x_1, x_2, x_3\} \times \{y, z_t\})$. Using that $K'$ is rainbow, we conclude that $\{x_1, x_2, x_3, y, z_t\}$ induces a rainbow copy of $K_5$. It remains to prove Claim~\ref{cl::B1}.

	\begin{proof}[Proof of Claim~\ref{cl::B1}]
		Let
		$$
			\T := \left\{X \in \binom{W}{3}: |N_G(X)| = \Omega_{d',\eps}(m)\right\}. 
		$$
		Then, $|\T| \ge \frac{1}{2}\binom{m}{3}$, owing to $G$ being $\eps$-regular with density $d'$. 
		We claim that the following properties hold a.a.s.
		\begin{enumerate}[label = \rm(\roman*)]
			\item \label{itm:K5a}
				$G_1$ admits a triangle whose vertex set is in $\T$.
			\item \label{itm:K5b}
				Every linear subset of vertices in $G_2$ spans a copy of $K_{1, 4}$.
		\end{enumerate}
		Indeed, Property \ref{itm:K5a} holds a.a.s.\ by Claim~\ref{cl::B2::Calc}, and Property \ref{itm:K5b} holds a.a.s.\ by Claim~\ref{cl::B4::Calc} (see Appendix~\ref{sec:Janson}). 

		Fix $G_1$ which satisfies Property \ref{itm:K5a} and $G_2$ which satisfies Property \ref{itm:K5b}. Then, by Property \ref{itm:K5a} there exists $X \in \T$ such that $G_1[X]$ is a triangle. Let $N = N_G(X)$ and note that $|N| = \Omega(m)$ holds by the definition of $\T$. Fix an arbitrary proper colouring $\psi$ of the edges of $\Gamma$. It follows by Observation~\ref{obs:compatible} that there exists a set $C_{\psi} \subseteq N$ of size $\Omega(m)$ which is compatible with $G_1[X]$ with respect to $\psi$. It follows by Property \ref{itm:K5b} that $G_2[C_{\psi}]$ spans a copy of $K_{1, 4}$, concluding the proof of Claim~\ref{cl::B1}.
	\end{proof}

\section{Rainbow copies of $K_6$}\label{sec:K6}

	In this section we prove Theorem~\ref{thm:main:6}. 

	\subsection{$1$-statement}\label{sec:K6:1-statement}

		Let $d \in (0, 1]$ be fixed, let $p := p(n) = \omega(n^{-2/3})$, and let $n$ be sufficiently large. Fix an arbitrary graph  $G \in \sG_{d, n}$.

		By a standard application of the (dense) regularity lemma~\cite{Szemeredi78} (see also~\cite{KS96}), we may assume that $G$ is an $\eps$-regular bipartite graph of edge-density $d'$ with bipartition $\{U, W\}$ satisfying $|U| = |W| = m = \Theta_{d',\eps}(n)$, for some sufficiently small constants $\eps, d' > 0$.
		Let $G_1 = (\Gnp)[W]$ and $G_2 = (\Gnp)[U]$, and consider the graph $\Gamma = G \cup G_1 \cup G_2$. We will show that a.a.s.\ $\Gamma \rainbow K_6$, thus proving the $1$-statement of Theorem~\ref{thm:main:6}.

		\medskip

		For reasons which will become apparent later on, we consider two fixed graphs, $R$ and $T_{10}$, defined as follows.
		Let $R_7$ denote the graph obtained from $K_{1,2}$ by attaching two triangles to each of its edges; that is, $V(R_7) = \{u_1, u_2, u_3, w_1, w_2, w_3, w_4\}$ and 
		\begin{equation} \label{eq:def-R7}	
			E(R_7) = \{u_1 u_2, u_2 u_3, u_1 w_1, u_1 w_2, u_2 w_1, u_2 w_2, u_2 w_3, u_2 w_4, u_3 w_3, u_3 w_4\}
		\end{equation} 
		(see Figure~\ref{fig:R}). Let $R$ be the vertex-disjoint union of $31$ copies of $R_7$.
		Next, let $T_k$ be obtained by \emph{gluing} $k$ vertex-disjoint triangles along a single (central) vertex (see Figure~\ref{fig:T}); that is 
		\begin{equation} \label{eq:def-T10}	
			T_{k} 
			:= (\{x, v_1, \ldots, v_{2k}\}, \{x v_i : 1 \leq i \leq 2k\} \cup \{v_{2i-1} v_{2i} : 1 \leq i \leq k\}).
		\end{equation} 

		\begin{figure}[ht]
			\centering
			\begin{subfigure}[b]{.3\textwidth}
				\centering
				\includegraphics[scale = 1]{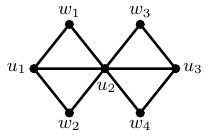}
				\caption{$R_7$}
				\label{fig:R}
			\end{subfigure}
			\hspace{1cm}
			\begin{subfigure}[b]{.3\textwidth} 
				\centering
				\includegraphics[scale = 1]{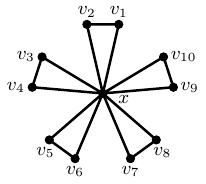}
				\caption{$T_5$}
				\label{fig:T}
			\end{subfigure}
			\caption{$R_7$ and $T_5$}
			\label{fig:RT}
		\end{figure}

		\begin{claim} \label{clm:oneK6}
			Asymptotically almost surely for every proper colouring $\psi$ of the edges of $\Gamma$, there is a copy of $K := K_{R, T_{10}}$ such that $V(\tilde{T})$ is compatible with $\tilde{R}$ with respect to $\psi$, where $\tilde{R}$ and $\tilde{T}$ are the natural embeddings of $R$ and $T_{10}$ in $K$.
		\end{claim}

		\begin{proof}
			Let
			$$
				\Z := 
				\left\{X \in \binom{W}{217}: |N_G(X)| = \Omega_{d',\eps}(m)\right\}. 
			$$
			Then, owing to $G$ being $\eps$-regular with edge-density $d'$ and to our assumption that $\eps$ is sufficiently small, it follows that $|\Z| \ge \frac{1}{2}\binom{m}{217}$. 
			We claim that the following two properties hold a.a.s.
			\begin{description}
				\item [(i)] $G_1$ admits a copy of $R$ whose vertex set is in $\Z$.
				\item [(ii)] Every linear subset of vertices in $G_2$ spans a copy of $T_{10}$.
			\end{description}
			Indeed, Property (i) follows from Claim~\ref{cl::31R7inAgoodSet::Calc}, and Property (ii) follows from Claim~\ref{cl::T10onTheRight::Calc} (see Appendix~\ref{sec:Janson}).

			Fix $G_1$ which satisfies Property (i) and $G_2$ which satisfies Property (ii), and let $\psi$ be a proper colouring of the edges of $\Gamma$. Then there exists $X \in \Z$ such that $G_1[X]$ spans a copy of $R$. Let $N := N_G(X)$ and note that $|N| = \Omega(m)$ holds by the definition of $\Z$. It follows by  Observation~\ref{obs:compatible} that there exists a set $C_\psi \subseteq N$ of size $\Omega(m)$ which is compatible with $G_1[X]$ with respect to $\psi$. Owing to Property (ii), the graph $G_2[C_\psi]$ admits a copy of $T_{10}$; denote its vertex set by $Y$. Then $X \cup Y$ spans a copy of $K$ with the required property (with respect to $\psi$).
		\end{proof}

		Fix $G_1$ and $G_2$ which satisfy the assertion of Claim~\ref{clm:oneK6}, and let $\psi$ be a proper colouring of the edges of $\Gamma$. Then there is a copy of $K_{R, T_{10}}$ whose vertex-set is $X \cup Y$, where $X$ spans a copy $\tilde{R}$ of $R$; $Y$ spans a copy of $T_{10}$; $\Gamma[X, Y]$ is complete and rainbow (under $\psi$); and the colours appearing on the edges of $\tilde{R}$ are not used for any edge of $\Gamma[X, Y]$. The following claim will be used to prove that $\Gamma[X \cup Y]$ admits a rainbow copy of $K_6$.

		\begin{claim} \label{cl::T10R7}
			Let $\psi$ be a proper edge-colouring of a vertex-disjoint union of $R_7$ and $T_{10}$. Then there exist triangles $Q_1 \subseteq T_{10}$ and $Q_2 \subseteq R_7$ such that $\psi(Q_1) \cap \psi(Q_2) = \emptyset$. 
		\end{claim}

		\begin{proof}
			Let $V(R_7) = \{u_1, u_2, u_3, w_1, w_2, w_3, w_4\}$ and let $E(R_7)$ be defined as in \eqref{eq:def-R7}. Let $V(T_{10}) = \{x, v_1, \ldots, v_{20}\}$ and let $E(T_{10})$ be defined as in \eqref{eq:def-T10}.
			Since $\psi$ is proper, every colour appears at most once on an edge incident with $u_2$. Hence, we may assume without loss of generality that
			$$
				\psi(u_1 u_2) = 1,\,\, 
				\psi(u_2 w_1) = 2,\,\,
				\psi(u_2 w_2) = 3,\,\,
				\psi(u_2 w_3) = 4,\,\,
				\psi(u_2 w_4) = 5,\,\,
				\psi(u_2 u_3) = 6.
			$$ 
			Write
			$$
				\psi(u_1 w_1) = \alpha,\,\,
				\psi(u_1 w_2) = \beta,
			$$
			and note that $\alpha \neq \beta$.
			Similarly, every colour appears at most once on an edge incident with $x$, and thus, without loss of generality, we may assume that 
			$$
				\{\psi(x v_{2i-1}), \psi(x v_{2i})\} \cap \{1,2,3,4,5,6\} = \emptyset
			$$ 
			holds for every $i \in [4]$. Write
			$$ 
				\psi(v_{2i-1} v_{2i}) = \gamma_i
			$$
			for every $i \in [4]$. 

			Suppose first that $\gamma_i = \gamma_j$ for some distinct $i,j \in [4]$. Assume without loss of generality that $\gamma_1 = \gamma_2 =: \gamma$, and $\gamma \notin [3]$ (the complementary case $\gamma \in [3] \Rightarrow \gamma \notin \{4,5,6\}$ can be treated similarly).
			Moreover, as $\alpha \neq \beta$, without loss of generality $\gamma \neq \alpha$. Since $\psi(xv_i)$ are distinct for $i \in [4]$, without loss of generality $\alpha \notin \{\psi(xv_1), \psi(xv_2)\}$. We may thus pick $Q_1 = x v_1 v_2$ and $Q_2 = u_1 u_2 w_1$.

			Next, we may assume that $\gamma_1, \gamma_2, \gamma_3, \gamma_4$ are distinct.
			Without loss of generality, $\{\gamma_1, \gamma_2, \gamma_3, \gamma_4\}$ contains at most one of $1$ and $2$ (otherwise, it contains at most one of $1$ and $3$ or at most one of $4$ and $6$ and these cases can be treated similarly). It follows that at most one of the triangles $x v_1 v_2, x v_3 v_4, x v_5 v_6, x v_7 v_8$ has an edge coloured $1$ or $2$. Moreover, at most two of these triangles contain an edge coloured $\alpha$. Thus, one of these triangles does not have edges coloured $1$, $2$, or $\alpha$; take $Q_1$ to be such a triangle and let $Q_2 = u_1 u_2 w_1$.
		\end{proof}

		With Claim~\ref{cl::T10R7} at hand, we prove that $\Gamma[X \cup Y]$ admits a rainbow copy of $K_6$. Let $X_1, \ldots, X_{31}$ denote the vertex sets of pairwise vertex-disjoint copies of $R_7$ in $\Gamma[X]$. By Claim~\ref{cl::T10R7}, for every $i \in [31]$ there are triangles $Q'_i \subseteq \Gamma[X_i]$ and $Q''_i \subseteq \Gamma[Y]$ such that $\psi(Q'_i) \cap \psi(Q''_i) = \emptyset$. Hence, there are four pairwise vertex-disjoint triangles $Q_1, Q_2, Q_3, Q_4 \subseteq \Gamma[X]$ and a triangle $Q \subseteq \Gamma[Y]$ and such that 
		\begin{equation} \label{eq::4triangles}
			\psi(Q_i) \cap \psi(Q) = \emptyset \textrm{ for every } i \in [4]. 
		\end{equation}
		Since $\psi$ is proper, there exists an $i \in [4]$ such that $\psi(V(Q_i) \times V(Q)) \cap \psi(Q) = \emptyset$. Since, by assumption, $\Gamma[V(Q_i), V(Q)]$ is rainbow under $\psi$ and $\psi(Q_i) \cap \psi(V(Q_i) \times V(Q)) = \emptyset$, it follows by~\eqref{eq::4triangles} that $\Gamma[V(Q_i) \cup V(Q)]$ is a rainbow copy of $K_6$.

	\subsection{$0$-statement}

		In this section we prove the second part of Theorem~\ref{thm:main:6} asserting that 
		for every $0 < d \leq 1/2$ and every $\eps > 0$, a.a.s.\ $\sG_{d,n} \cup \mathbb{G}(n,p) \notrainbow K_6$, whenever $p:= p(n) = n^{-(2/3+\eps)}$. We deduce this from the following lemma which is the main result of this section. 

		\begin{lemma}\label{lem:4-matchings}
		For every $\eps > 0$ and $p := p(n) = n^{-(2/3+\eps)}$, a.a.s.\ $R \sim \mathbb{G}(n,p)$ contains four pairwise edge-disjoint matchings, namely $M_0, M_1, M_2, M_3$, such that the following holds. 
		\begin{enumerate}
			\item 
				$M_0$ and $M_i$ are vertex-disjoint for every $i \in [3]$; and 
			\item 
				every triangle in $R$ either contains an edge of $M_0$ or contains edges from at least two of the matchings $M_1, M_2, M_3$. 
		\end{enumerate}
		\end{lemma}

		Prior to proving Lemma~\ref{lem:4-matchings}, we use it to derive the aforementioned $0$-statement for the emergence of rainbow copies of $K_6$ in the perturbed model, i.e., the second part of Theorem~\ref{thm:main:6}. 

		\begin{proof}[Proof of the $0$-statment for $K_6$ using Lemma~\ref{lem:4-matchings}]
			Let $\eps > 0$ be fixed and let $p := p(n) = n^{-(2/3+\eps)}$. 
			Then, $R \sim \mathbb{G}(n,p)$ is a.a.s.\  $K_4$-free (as the expected number of copies of $K_4$ is $O(n^4p^6) = o(1)$). Let a $K_4$-free graph $R$, satisfying the assertion of Lemma~\ref{lem:4-matchings}, be fixed. We prove that $\Gamma := G \cup R$ satisfies $\Gamma \notrainbow K_6$, where $G \cong K_{\floor{n/2},\ceil{n/2}}$ is a balanced complete bipartite graph with bipartition $V(G) = A \discup B$. 

			\medskip

			Define an assignment of colours to the edges of $\Gamma$ as follows. 
			\begin{enumerate}[leftmargin = *, label = \bf (C\arabic*)]
				\item  
					Colour the edges of the matchings $M_0$ and $M_1$ (found in $R$) red. Colour the edges of $M_2$ blue; and colour the edges of $M_3$ green.

				\item 
					Given an unordered pair of edges $xy \in M_0$ and $zw \in M_2$ such that either $\{x,y\} \subseteq A$ and $\{z,w\} \subseteq B$, or $\{x,y\} \subseteq B$ and $\{z,w\} \subseteq A$, the members of $E_G(\{x,y\},\{z,w\})$ define a copy of $C_4$ in $G$. Colour the members of $E_G(\{x,y\},\{z,w\})$ using two colours that are unique to the pair $\{xy, zw\}$ (i.e., the colours have never been used before on any other edge coloured thus far) and in such a way that a proper edge colouring is defined over the copy of $C_4$ arising from $E_G(\{x,y\},\{z,w\})$. 
				
				\item
					Colour the remaining uncoloured edges of $\Gamma$ distinctively; each with its unique new colour. 
			 \end{enumerate}

			Let $\psi$ be the resulting colour assignment. First, observe that $\psi$ is a well-defined edge-colouring of $\Gamma$. It is clear that each edge of $\Gamma$ is assigned at least one colour. Owing to $M_0, M_1, M_2, M_3$ being pairwise edge-disjoint and  owing to $M_0$ and $M_2$ being vertex disjoint, no edge of $\Gamma$ is assigned more than one colour. Next, note that $\psi$ is a proper edge-colouring of $\Gamma$. For the edges coloured red, this holds as $M_0$ and $M_1$ are vertex-disjoint. For all other colours this is self-evident. 

			\medskip

			It remains to prove that no $\psi$-rainbow copy of $K_6$ exists in $\Gamma$. To this end, let a copy of $K_6$ in $\Gamma$, denoted $K$, be fixed. As $R$ is $K_4$-free, the set $V(K)$ is comprised of three vertices from $A$ and the other three from $B$; each such triple forming a triangle in $R$. Let $T \subseteq R[A]$ and $S \subseteq R[B]$ denote these two triangles. Since $R$ satisfies the property described in Lemma~\ref{lem:4-matchings}, at least one of the following alternatives holds. 
			\begin{enumerate}[leftmargin = *, label = \bf (A\arabic*)]
				\item \label{itm:A1} 
					Both $T$ and $S$ contain an edge from $M_0$.
				\item \label{itm:A2}
					Both $T$ and $S$ contain edges from two of $M_1, M_2, M_3$. 
				\item \label{itm:A3}
					$T$ contains an edge from $M_0$ and $S$ contains an edge from $M_1$ (or 
				vice versa).
				\item \label{itm:A4} 
					$T$ contains an edge from $M_0$ and $S$ contains an edge from $M_2$ and an edge from $M_3$ (or vice versa).
			\end{enumerate}
			If one of \ref{itm:A1}, \ref{itm:A2}, \ref{itm:A3} holds, then the triangles $T$ and $S$ have a colour (red, blue or green) in common. If \ref{itm:A4} holds, then there are two edges of the same colour between $T$ and $S$.
			Either way, the $K_6$-copy $K$ is not $\psi$-rainbow, as required.
		\end{proof}

		It remains to prove Lemma~\ref{lem:4-matchings}. 

		\begin{proof}[Proof of Lemma~\ref{lem:4-matchings}]
			Fix $\eps > 0$. Given $R \sim \mathbb{G}(n, p)$, let $R'$ be the subgraph of $R$ which is the union of all triangles in $R$. It suffices to prove that a.a.s.\ the required matchings exist for every connected component of $R'$.

			Given a connected component $F$ of $R'$, let $F_0, F_1, \ldots, F_\ell$ be a (nested) sequence of connected subgraphs of $F$ defined (recursively) as follows. The \emph{starting} graph, namely $F_0$, is an arbitrary copy of $K_3$ in $F$. Suppose that $F_0, \ldots, F_{i-1}$ have already been defined. If $F_{i-1} = F$ or if $i-1 > 1/\eps$, stop and set $\ell := {i-1}$. Otherwise, since $F$ is connected, there is an edge $e_i = x_i y_i \in E(F) \setminus E(F_{i-1})$ such that $x_i \in V(F_{i-1})$. Let $z_i \in V(F)$ be a vertex such that the set $\{x_i, y_i, z_i\}$ forms a triangle in $F$ (such a $z_i$ exists by the definition of $R'$). 
			Then, one of the following alternatives holds (up to relabelling).
			\begin{enumerate}[label = \rm{(\alph*)}] 
				\item \label{itm:step-a} 
					$x_i \in V(F_{i-1})$, $y_i, z_i \notin V(F_{i-1})$.
				
				\item \label{itm:step-b} 
					$x_i, z_i \in V(F_{i-1})$, $y_i \notin V(F_{i-1})$, and $x_iz_i \in E(F_{i-1})$.

				\item $x_i, z_i \in V(F_{i-1})$, $y_i \notin V(F_{i-1})$, and $x_iz_i \notin E(F_{i-1})$.
				
				\item \label{itm:step-d} 
				$x_i, y_i, z_i \in V(F_{i-1})$, and $y_i z_i, x_i z_i \in E(F_{i-1})$.
				
				\item $x_i, y_i, z_i \in V(F_{i-1})$, and at least one of $y_i z_i, x_i z_i$ is not in $E(F_{i-1})$.
			\end{enumerate} 
			Define $F_i$ to be the subgraph of $F$ with vertex set $V(F_{i-1}) \cup \{y_i, z_i\}$ and edge set $E(F_{i-1}) \cup \{x_i y_i, x_i z_i, y_i z_i\}$.

			\medskip

			Write $\alpha, \beta, \gamma, \delta, \zeta$ to denote the number of values $i$ for which the first, second, third, fourth, and fifth alternative held throughout the construction of the sequence, respectively. 
			Then, 
			$$
			v := v(F_\ell) = 3 + 2\alpha + \beta + \gamma \; \text{and} \; e := e(F_\ell)  \geq 3 + 3\alpha + 2\beta + 3\gamma + \delta + 2\zeta.
			$$ 
			Given values of $\alpha, \beta, \gamma, \delta, \zeta$ whose sum is at most $1/\eps + 1$, there are $O_\eps(1)$ possible configurations for the terminating graph $F_\ell$. For any single such configuration $C$, the expected number of copies of $C$ in $\mathbb{G}(n, p)$ is at most 
			\[
				O\!\left(n^v p^e \right) =
				O\! \left( n^{3 + 2 \alpha + \beta + \gamma - (2/3 + \eps) \cdot (3 + 3\alpha + 2 \beta + 3\gamma + \delta + 2\zeta)} \right) =
				O \! \left( n^{1 - \beta/3 - \gamma - 2\delta/3 - 4\zeta/3 - \ell \eps} \right),
			\]
			where in the last equality we use the fact that $e \geq \ell$, entailing the term $\eps \ell$ appearing in the exponent. 
			We may assume that $1 - \beta/3 - \gamma - 2\delta/3 - 4\zeta/3 - \ell \eps \ge 0$, for otherwise there are no copies of $C$ in $G$ a.a.s.\ across all of its possible configurations $C$ with values $\alpha, \beta, \gamma, \delta, \zeta$, owing to Markov's inequality and the fact that the number of possible configurations is $O_\eps(1)$. As $\ell \eps > 0$, it follows that
			$$
				\gamma = \zeta = 0, \quad
				\ell \leq 1/\eps, \quad
				\beta \in \{0, 1, 2\}, \quad
				\delta \in \{0, 1\}, \quad
				\beta + 2\delta \in \{0, 1, 2\}.
			$$ 
			The fact that $\ell \leq 1/\eps$ implies that, by definition, the sequence terminated due to $F_\ell$ coinciding with $F$ so that $F_{\ell} = F$ holds. 

			\medskip

			We may assume, without loss of generality, that $\delta = 0$. Indeed, otherwise $\delta = 1$ and thus $\beta = 0$. It then follows that there is one step of type \ref{itm:step-d} and all other steps are of type \ref{itm:step-a}, implying that the graph \raisebox{-4pt}{\figone{scale=.6}} is a subgraph of $F$. This in turn means that the sequence could have started with two steps of type \ref{itm:step-b}, i.e.,\ that $\beta \geq 2$ and thus $\delta = 0$. 

			In what follows we construct the required matchings via a case analysis ranging over the three possible values of $\beta$.
			\begin{description}
				\item [Case I: $\beta = 0$.] 
					In this case all steps are of type \ref{itm:step-a}. Take $M_0$ to be a matching that consists of some edge in $F_0$, and the edges $\{y_i z_i : i \in [\ell]\}$ and let $M_1 = M_2 = M_3 = \emptyset$. It is self-evident that, in this case, $M_0$ is a matching meeting all triangles of $F$.

				\item [Case II: $\beta = 1$.] 
					In this case \raisebox{-4pt}{\figtwo{scale=.7}} 
					is a subgraph of $F$. We may thus assume that the first step is of type \ref{itm:step-b}, and all other steps are of type \ref{itm:step-a}. Let $M_0$ be the matching consisting of the edges $x_1 z_1$ and $\{y_i z_i : 2 \leq i \leq \ell\}$ and let $M_1 = M_2 = M_3 = \emptyset$. Then, again, $M_0$ is a matching meeting all triangles of $F$. 


				\item [Case III: $\beta = 2$.] 
					In this case $F$ can be formed by making steps of type \ref{itm:step-a}, starting with one of the graphs \raisebox{-4pt}{\figthree{scale=0.6}} or \raisebox{-4pt}{\figone{scale=0.6}}, or there are two edge-disjoint copies of \raisebox{-4pt}{\figtwo{scale=0.7}}. In the former case, one can verify that there exists a matching $M_0$ meeting all triangles of $F$, by finding such a matching in the starting graph and extending it by adding the edges $y_i z_i$. 
					In the latter case, $F$ can be formed by making steps of type \ref{itm:step-a}, starting with one of the families of graphs depicted in Figure~\ref{fig:three-types}.

					\begin{figure}[ht] 
						\centering
						\begin{subfigure}[b]{.4\textwidth}
							\centering
							\includegraphics[scale = 1.5]{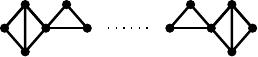}
							\caption{Type I}
							\label{fig:typeI}
						\end{subfigure}
						\hspace{1cm}
						\begin{subfigure}[b]{.4\textwidth}
							\centering
							\includegraphics[scale = 1.5]{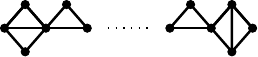}
							\caption{Type II}
							\label{fig:typeII}
						\end{subfigure}

						\vspace{.8cm}
						\begin{subfigure}[b]{.4\textwidth}
							\centering
							\includegraphics[scale = 1.5]{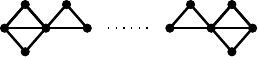}
							\caption{Type III}
							\label{fig:typeIII}
						\end{subfigure}

						\vspace{.2cm}
						\caption{Three families of starting graphs.}
						\label{fig:three-types}
					\end{figure}

					For each of the first two families of starting graphs, namely of type I (see Figure~\ref{fig:typeI}) and of type II (see Figure~\ref{fig:typeII}), there is a matching $M_0'$ meeting all of their triangles (see Figure~\ref{fig:two-types-coloured}); and this matching can be extended into a matching $M_0$ in $F$ meeting all triangles of $F$, by adding the edges of the form $y_i z_i$ defined in subsequent steps. As in previous cases, we set $M_1 = M_2 = M_3 = \emptyset$.

					\begin{figure}[ht] 
						\centering
						\begin{subfigure}[b]{.4\textwidth}
							\centering
							\includegraphics[scale = 1.5]{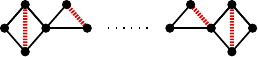}
							\caption{Colouring type I}
							\label{fig:typeI-col}
						\end{subfigure}
						\hspace{1cm}
						\begin{subfigure}[b]{.4\textwidth}
							\centering
							\includegraphics[scale = 1.5]{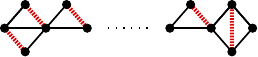}
							\caption{Colouring type II}
							\label{fig:typeII-col}
						\end{subfigure}
						\vspace{.2cm}

						\caption{Colouring the first two starting graphs}
						\label{fig:two-types-coloured}
					\end{figure}

					For the third family of starting graph, of type III (see Figure~\ref{fig:typeIII}), there are three edge-disjoint matchings $M_1, M_2, M_3$ such that every triangle of the starting graph contains edges from at least two of these matchings (see Figure~\ref{fig:bad-type}). In this case, set $M_0$ to consist of the edges of the form $y_i z_i$ defined in subsequent steps.

					\begin{figure}[ht] 
						\centering
						\includegraphics[scale = 2.5]{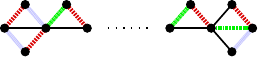}

						\caption{Colouring type III graphs}
						\label{fig:bad-type}
					\end{figure}
								 
					One may readily check that the matchings $M_0, M_1, M_2, M_3$, defined above, satisfy the properties stipulated in Lemma~\ref{lem:4-matchings}. \qedhere
			\end{description}
		\end{proof}

\section{Rainbow copies of $K_7$}\label{sec:K7}

	In this section we prove the third part of Theorem~\ref{thm:main:457}. That is, we prove that the threshold for the property $\sG_{d,n} \cup \mathbb{G}(n,p) \rainbow K_7$ is $n^{-7/15}$. To see the $0$-statement, fix some $d \leq 1/2$ and let $G$ be a bipartite graph on $n$ vertices with density $d$, and let $p = o \left(n^{-7/15} \right)$. Since $G$ is bipartite, any rainbow copy of $K_7$ in $\Gamma \sim G \cup \Gnp$ must contain a rainbow copy of $K_4$ in $\Gnp$. However, as proved in~\cite{KMPS18}, a.a.s.\ the property $\mathbb{G}(n,p) \rainbow K_4$ does not hold whenever $p = o(n^{-7/15})$.
	

	\medskip

	Proceeding to the $1$-statement, let $d \in (0, 1]$ be fixed, let $p := p(n) = \omega (n^{-7/15})$, and let $n$ be sufficiently large. Let $G \in \sG_{d, n}$.

	By a standard application of the (dense) regularity lemma~\cite{Szemeredi78} (see also~\cite{KS96}), we may assume that $G$ is an $\eps$-regular bipartite graph of edge-density $d'$ with bipartition $\{U, W\}$ satisfying $|U| = |W| = m = \Theta_{d',\eps}(n)$, where $\eps, d' > 0$ are sufficiently small constants. Let $G_1 = (\Gnp)[W]$ and $G_2 = (\Gnp)[U]$, and set $\Gamma = G \cup G_1 \cup G_2$. We will show that a.a.s. $\Gamma \rainbow K_7$.

	Let $H$ be the disjoint union of four copies of $\widehat{K}_{3,4}$ (recall that $\widehat{K}_{3, 4}$ is the join of a triangle and an independent set of size four, as defined in Section~\ref{sec:sparse-bip}), and let $F$ be the graph obtained from $K_{1, 25}$ by attaching $49$ triangles to each of its edges, where the vertex not in $K_{1,25}$ is unique to each triangle. The copy of $K_{1,25}$ giving rise to $F$ is referred to as its \emph{skeleton}. 

	\begin{claim} \label{clm:oneK7}
		Asymptotically almost surely for every proper colouring $\psi$ of $\Gamma$ there is a copy of $K := K_{H, F}$ such that $V(\tilde{F})$ is compatible with $\tilde{H}$ with respect to $\psi$, where $\tilde{H}$ and $\tilde{F}$ are the natural embeddings of $H$ and $F$ in $K$.
	\end{claim}

	\begin{proof}
		Let
		$$
			\Z := \left\{X \in \binom{W}{28}: |N_G(X)| = \Omega_{d',\eps}(m)\right\}. 
		$$
		Owing to $G$ being $\eps$-regular with edge-density $d'$ and to our assumption that $\eps$ is small, it follows that $|\Z| \ge \frac{1}{2} \binom{m}{28}$. 
		We claim that the following two properties hold a.a.s.
		\begin{description}
			\item [(i)] $G_1$ admits a copy of $H$ whose vertex set is in $\Z$.
			
			\item [(ii)] Every linear subset of vertices in $G_2$ spans a copy of $F$.
		\end{description}
		Indeed, Property (i) holds a.a.s.\ by Claim~\ref{cl::C1::Calc}, and Property (ii) holds a.a.s.\ by Claim~\ref{clm::C3::Calc} (see Appendix~\ref{sec:Janson}).

		Fix $G_1$ which satisfies Property (i) and $G_2$ which satisfies Property (ii), and fix a proper colouring $\psi$ of the edges of $\Gamma$. By Property (i) there exists a set $X \in \Z$ such that $G_1[X]$ spans a copy of $H$. Let $N = N_G(X)$ and note that $|N| = \Omega(m)$ holds by the definition of $\Z$.
		It follows by Observation~\ref{obs:compatible} that there exists a set $C_\psi \subseteq N$ of size $\Omega(m)$ which is compatible with $G_1[X]$ with respect to $\psi$. Owing to Property (ii), the graph $G_2[C_{\psi}]$ admits a copy of $F$; denote its vertex set by $Y$. Then $X \cup Y$ spans a copy of $K$ with the required property (with respect to $\psi$). 
	\end{proof}

	Fix $G_1$ and $G_2$ which satisfy the assertion of Claim~\ref{clm:oneK7}, and let $\psi$ be a proper colouring of the edges of $\Gamma$. Therefore, there exists a copy of $K_{H, F}$ with vertex-set $X \cup Y$, where $X$ spans a copy $\tilde{H}$ of $H$; $Y$ spans a copy $\tilde{F}$ of $F$; $\Gamma[X, Y]$ is complete and rainbow (under $\psi$); and the colours appearing on the edges of $\tilde{H}$ are not used for any edge of $\Gamma[X, Y]$.

	It is easy to verify that $\widehat{K}_{3,4} \rainbow K_4$ (this was also observed in~\cite{KMPS18}). Consequently, $X$ admits four pairwise vertex-disjoint rainbow copies of $K_4$; denote their vertex-sets by $X_1, X_2, X_3, X_4$ and write $X' := X_1 \cup \ldots \cup X_4$ and $H' = \Gamma[X_1] \cup \ldots \cup \Gamma[X_4]$. 
	\medskip

	In order to complete the proof of the 1-statement for $K_7$, we prove that $\Gamma[X' \cup Y]$ admits a $\psi$-rainbow copy of $K_7$. Observe that $\Gamma[A \cup B] \cong K_7$ for every $A \in \{X_1, \ldots, X_4\}$ and $B \subseteq V(Y)$ such that $\Gamma[B] \cong K_3$. Since $\Gamma[A] \cup \Gamma[A, B]$ is $\psi$-rainbow for all such choices of $A$ and $B$, if $\Gamma[A \cup B]$ is not rainbow, then there exist edges $e_A \in E_{\Gamma}(A) \cup E_{\Gamma}(A,B)$ and $e_B \in E_{\Gamma}(B)$ such that $\psi(e_A) = \psi(e_B)$. Dealing with the case $e_A \in E_{\Gamma}(A)$ first, we delete from $\tilde{F}$ every edge whose colour under $\psi$ appears in $\psi(E(H'))$. Owing to $\psi$ being proper, this entails the removal of at most $24$ matchings from $\tilde{F}$. We claim that this does not destroy all of the triangles of $\tilde{F}$.

	\begin{observation}\label{obs::C4}
		The removal of any $24$ matchings from $F$ yields a graph which is not triangle-free.
	\end{observation}

	\begin{proof}
		Let $M_1, \ldots, M_{24}$ be any $24$ matchings in $F$ and let $F' = F \setminus (M_1 \cup \ldots \cup M_{24})$. At least one of the edges of the skeleton of $F$, say $e$, is retained in $F'$. Observe that, for every $i \in [24]$, the matching $M_i$ meets the edges of at most two of the triangles of $F$ associated with $e$. Therefore, at least one of the $49$ triangles associated with $e$ in $F$ remains intact in $F'$.
	\end{proof}

	Following Observation~\ref{obs::C4}, let $T \subseteq \tilde{F}$ be a triangle that has persisted the removal of all edges of $\tilde{F}$ that were assigned a colour which appears in $\psi(E(H'))$. It thus remains to take care of colour clashes between the edges of $T$ and the edges connecting it to $X'$. For every $i \in [4]$, let $E_i = E_\Gamma(X_i, V(T))$. Since $\psi$ is proper, if $\psi(E(T)) \cap \psi \left(\cup_{i=1}^4 E_i \right) \neq \emptyset$, then there are two independent edges $e \in E(T)$ and $e' \in \cup_{i=1}^4 E_i$ such that $\psi(e) = \psi(e')$. Since $T$ is a triangle and $\psi$ is proper, there are at most three such pairs of edges. Consequently, there exists an index $i^* \in [4]$ such that $\psi(E(T)) \cap \psi(E_{i^*}) = \emptyset$. Then, $\Gamma[X_{i^*} \cup V(T)] \cong K_7$ is rainbow under $\psi$.

\section{Rainbow copies of $K_8$}\label{sec:K8}

	In this section, we prove Theorem~\ref{thm:main:8}. That is, we prove that given $0 < d \leq 1/2$ and $\eps >0$, the property $\sG_{d,n} \cup \mathbb{G}(n,p) \notrainbow K_8$ holds a.a.s., whenever $p:= p(n) = n^{-(2/5+\eps)}$. The following implies Theorem~\ref{thm:main:8}. 

	\begin{proposition}\label{thm:cover-rainbow-K4}
		Let $\eps > 0$ and let $p = n^{-2/5 - \eps}$. Then, a.a.s.\ the edges of $\mathbb{G}(n, p)$ can be properly coloured so that all rainbow copies of $K_4$ share at least one common colour. 
	\end{proposition}

	Prior to proving Proposition~\ref{thm:cover-rainbow-K4}, we use it to deduce Theorem~\ref{thm:main:8}. 

	\begin{proof}[Proof of Theorem~\ref{thm:main:8} using Proposition~\ref{thm:cover-rainbow-K4}]
		Fix $G \sim \mathbb{G}(n, p)$ satisfying the property specified in Proposition~\ref{thm:cover-rainbow-K4}. Then $G$ admits a proper edge-colouring $\psi$ such that all copies of $K_4$ in $G$ which are rainbow under $\psi$ contain an edge coloured, say, red. This further implies that $\psi$ gives rise to no rainbow copy of $K_5$. Indeed, suppose the vertex set $\{a,b,c,d,e\}$ induces a rainbow copy of $K_5$, then there is an edge of that copy, say $ab$, which is coloured red. Then, the vertex set  $\{b,c,d,e\}$ induces a rainbow copy of $K_4$ without a red edge, a contradiction. 

		Given an $n$-vertex bipartite graph $B$, extend the edge-colouring $\psi$ into a proper edge-colouring of $G \cup B$ arbitrarily, and let $\psi'$ denote the resulting colouring. Let $K$ be a copy of $K_8$ in $G \cup B$, and let $K'$ and $K''$ denote the intersections of $K$ with the two parts of the bipartition of $B$. We may assume that both $K'$ and $K''$ are $\psi$-rainbow, for otherwise $K$ is clearly not $\psi'$-rainbow. As $\psi$ does not give rise to any rainbow copies of $K_5$ in $G$, it follows that $K', K'' \cong K_4$. 
		Then, while $\psi$-rainbow on their own, $K'$ and $K''$ have a colour in common and the proof follows. 
	\end{proof}

	The remainder of this section is dedicated to the proof of Proposition~\ref{thm:cover-rainbow-K4}. In Section~\ref{sec:generating} we introduce some useful terminology. In Section~\ref{sec:proof-cover-rainbow-K4} we deduce Proposition~\ref{thm:cover-rainbow-K4} from the main result of this section, namely Lemma~\ref{lem:phi}, stated below. In Section~\ref{sec:classification}, we prove Lemma~\ref{lem:phi}. 

	\subsection{Stretched generating sequences and their properties} \label{sec:generating}

		For a graph $H$, let $\K_4(H)$ be the auxiliary graph whose vertices are the copies of $K_4$ in $H$, with two such copies being adjacent if and only if they are not edge-disjoint. We say that $H$ is $K_4$-{\em connected} if $\K_4(H)$ is connected. Moreover, we say that $H$ is $K_4$-{\em covered} if every edge of $H$ lies in some  copy of $K_4$. Graphs $H$ that are both $K_4$-connected and $K_4$-covered are called $K_4$-{\em tiled}. Such graphs can be {\em generated} through a (nested) sequence of connected subgraphs of $H$, namely 
		$$
		H_0 \cong K_4, H_1, \ldots, H_r = H,
		$$
		such that for every $i \in [r]$, the graph $H_i$ can be obtained from $H_{i-1}$ using one of the following steps.
		\begin{description}
			\item[Standard steps.] \label{itm:step-two-vs}
				Let $z_i w_i \in E(H_{i-1})$ and let $x_i, y_i \in V(H) \setminus V(H_{i-1})$ be distinct. Define $H_i$ by setting
				$$
				V(H_i) := V(H_{i-1}) \cup \{x_i, y_i\}\; \text{and}\; E(H_i) := E(H_{i-1}) 
				\cup \{x_i y_i, x_i z_i, x_i w_i, y_i z_i, y_i w_i\}.
				$$
			\item[Vertex-steps.] \label{itm:step-one-vx}
				Let $y_i, z_i, w_i \in V(H_{i-1})$ be distinct vertices that span at least one edge of $H_{i-1}$. Let $x_i \in V(H) \setminus V(H_{i-1})$.  Define $H_i$ by setting 
				$$
				V(H_i) : = V(H_{i-1}) \cup \{x_i\}\; \text{and}\; E(H_i) := E(H_{i-1}) \cup 
				\{x_i y_i, x_i z_i, x_i w_i, y_i z_i, y_i w_i, w_i z_i\}.
				$$ 
				Such vertex-steps are further distinguished and are said to be {\em with} or {\em without} missing edges, according to whether or not at least one of the pairs $\{y_i, z_i\}$, $\{y_i, w_i\}$, and $\{z_i, w_i\}$ forms a non-edge of $H_{i-1}$, respectively. 

			\item[Edge-steps.] \label{itm:step-edge}
				Let $x_i, y_i, z_i, w_i \in V(H_{i-1})$ be distinct vertices that span between one and five edges. Define $H_i$ by setting 
				$$
				V(H_i) = V(H_{i-1})\; \text{and} \; E(H_i):= E(H_{i-1}) \cup \{x_i y_i, x_i 
				z_i, x_i w_i, y_i z_i, y_i w_i, w_i z_i\}.
				$$
				Edge-steps adding $m$ new edges are called \emph{$m$-edge-steps}.
		\end{description}
		Observe that in each of the above three step-types, the vertices of $\{x_i, y_i, z_i, w_i\}$ induce a copy of $K_4$ in $H_i$ but not in $H_{i-1}$; moreover, they span at least one edge in $H_{i-1}$. 

		Given a sequence generating $H$, let $\gamma$ denote the number of edges added throughout along edge-steps, and between existing vertices in vertex-steps with missing edges. 

		A $K_4$-tiled graph $H$ may admit numerous generating sequences. Sequences generating $H$ that
		\begin{enumerate}[label = \bf(T\arabic*)]
			\item \label{itm:min-gamma}
				minimise $\gamma$, and
			\item \label{itm:max-r}
				amongst generating sequences satisfying \ref{itm:min-gamma}, maximise the length of the sequence $r$,
		\end{enumerate}
		are said to be {\em stretched}. Such sequences have the property that the addition of the missing edges alone in vertex-steps (with missing edges) does not yield a new copy of $K_4$. For otherwise, one may split such a vertex-step into an edge-step followed by a vertex-step keeping $\gamma$ unchanged, yet increasing the length of the sequence; contrary to its maximality stated in~\ref{itm:max-r}. Similarly, adding any proper subset of the set of edges added in some edge-step does not give rise to a new copy of $K_4$; this would again contradict the maximality stated in~\ref{itm:max-r}. 

		\medskip
		The following claim facilitates our proof of Proposition~\ref{thm:cover-rainbow-K4}. Its proof can be found in Appendix~\ref{sec:proof-first-edge-step}.

		\begin{claim} \label{claim:first-edge-step}
			Let $H$ be a $K_4$-tiled $K_5$-free graph, and let $H_0 \cong K_4, H_1, \ldots, H_r = H$ be a stretched sequence generating $H$. Suppose that the first edge-step in the sequence is a $1$-edge-step that introduces the new edge $xy$, resulting in $\{x,y,z,w\}$ forming a copy of $K_4$. Then,
			\begin{enumerate}[label = \rm (\alph*)]
				\item \label{itm:edge-step-one-K4}
					If the first edge-step is not preceded by vertex-steps with missing edges, then $\{x,y,z,w\}$ is the sole new copy of $K_4$ incurred through the addition of the edge $xy$.
				\item \label{itm:edge-step-intersect-triangle}
					If the first edge-step is preceded by one vertex-step with one missing edge and no other vertex-steps with missing edges, then there is a triangle $T$ such that all the copies of $K_4$ that appear in the graph upon the addition of $xy$ contain $T$.
				\item \label{itm:edge-step-after-vx-step}
					The step introducing $xy$ is preceded by at least one vertex-step with missing edges, or at least two vertex-steps with no missing edges.
			\end{enumerate}
		\end{claim}

	\subsection{Proof of Proposition~\ref{thm:cover-rainbow-K4}}\label{sec:proof-cover-rainbow-K4} 

		For a $K_4$-tiled graph $H$ and a stretched sequence $H_0 \cong K_4, H_1, \ldots, H_r =H$ generating $H$, write $\alpha$ and $\beta$ to denote the number of standard steps and vertex-steps, respectively, taken throughout the sequence. Additionally, define $\gamma$, as in the previous section, to be the number of edges added throughout the sequence along vertex or edge-steps connecting two existing non-adjacent vertices. Then, 
		\begin{equation}\label{eq:H-edges-vxs}
			v(H) = 4 + 2 \alpha + \beta, \quad \text{and} \quad e(H) = 6 + 5\alpha + 3\beta + \gamma.
		\end{equation} 
		In particular, 
		\begin{equation}\label{eq:lower-e(H)-tiled}
			e(H) \geq (5/2)v(H) - 4.
		\end{equation}
		The parameter 
		$$
			\phi(H) := 8 - 5 v(H) + 2 e(H)  = 2\gamma + \beta
		$$
		will arise naturally in various calculations, (see e.g.~\eqref{eq:phi-arises}). Note that, by~\eqref{eq:lower-e(H)-tiled}, $\phi(H) \geq 0$ holds for every $K_4$-tiled graph; we will see below that a.a.s.\ $\phi(H) \leq 7$ holds for every $K_4$-tiled graph $H$ in $G \sim \mathbb{G}(n,p)$ with $p = n^{-(2/5 + \eps)}$ (see Claim~\ref{clm:tiled-sparse}). 
		A central ingredient in the proof of Proposition~\ref{thm:cover-rainbow-K4} is the following lemma, asserting the existence of certain proper edge-colourings of $K_4$-tiled graphs.  

		\begin{lemma} \label{lem:phi}
			Let $H$ be a $K_4$-tiled graph. 
			\begin{enumerate}[label = \rm (\roman*)]
				\item \label{itm:phi-1} 
					If $\phi(H) \in \{0, 1, 2\}$, then $H$ has a proper edge-colouring admitting no rainbow copies of $K_4$.

				\item \label{itm:phi-2}
					If $\phi(H) \in \{3,4,5\}$, then $H$ admits a triangle $T$ and a proper edge-colouring $\psi$ such that all rainbow copies of $K_4$ arising from $\psi$ contain $T$.

				\item \label{itm:phi-3}
					If $\phi(H) \in \{6, 7\}$, then $H$ admits a matching $M$ of size at most $3$ and a proper edge-colouring $\psi$ such that all rainbow copies of $K_4$ arising from $\psi$ meet $M$.  
			\end{enumerate}
		\end{lemma}

		\noindent
		The proof of Lemma~\ref{lem:phi} is postponed to Section~\ref{sec:classification}. The remainder of the current section is dedicated to the derivation of  Proposition~\ref{thm:cover-rainbow-K4} from this lemma. 

		\bigskip

		By a \emph{$K_4$-component} of a graph $G$, we mean a maximal $K_4$-tiled subgraph of $G$. Observe that such components are by definition pairwise edge-disjoint (recall the definition of the auxiliary graph $\K_4(G)$); yet they may have vertices in common. 

		The edge-set of a graph $G$ can be decomposed into a collection $\HH:=\HH(G)$ of (pairwise edge-disjoint) $K_4$-components, and a set $E'$ of edges of $G$ contained in no copy of $K_4$ in $G$. The members of $E'$ will be of no interest to us. 
		Owing to alternative \ref{itm:phi-1} of Lemma~\ref{lem:phi}, $K_4$-components $H$ satisfying $\phi(H) \leq 2$ are of no threat to us. It thus suffices to analyse the union of $K_4$-components $H$ satisfying $\phi(H) \geq 3$. 
		Given a graph $G$, consider the graph $G'$ which is the union of $K_4$-components $H$ of $G$, satisfying $\phi(H) \geq 3$, and let $\C := \C(G)$ be the collection of connected components in $G'$.

		\medskip

		Let $\eps > 0$ be given; note that we may assume that $\eps$ is arbitrarily small yet fixed. Set $p:= p(n) = n^{-(2/5+\eps)}$. Claims~\ref{clm:small-K_4-tiled} to~\ref{clm:K_4-collection-path}, stated below,  collectively capture properties that are a.a.s.\ satisfied simultaneously by $G \sim \mathbb{G}(n,p)$. Roughly speaking, these properties collectively assert that $K_4$-components $H$ of $G$, satisfying $\phi(H) \geq 3$, admit a tree-like structure. 


		\begin{claim}\label{clm:small-K_4-tiled}
			Asymptotically almost surely $G \sim \mathbb{G}(n,p)$ does not have $K_4$-tiled subgraphs on more than $\lceil 1/\eps \rceil$ vertices.
		\end{claim}

		\begin{proof}
			Owing to~\eqref{eq:lower-e(H)-tiled}, the expected number of $k$-vertex $K_4$-tiled subgraphs of $G \sim \mathbb{G}(n,p)$ is at most 
			\begin{equation}\label{eq:phi-arises}
				2^{k^2} \cdot n^k p^{(5/2)k - 4}
				= 2^{k^2} \cdot n^{k - (2/5 + \eps)((5/2)k - 4)}
				= 2^{k^2} \cdot n^{8/5 + 4 \eps - (5/2) \eps k}
				\leq 2^{k^2} \cdot n^{2 - (5/2) \eps k},
			\end{equation}
			where for the sole inequality above we use the fact that $\eps$ is arbitrarily small yet fixed. Consequently, by Markov's inequality, $G \sim \mathbb{G}(n,p)$ a.a.s.\ admits no $k$-vertex $K_4$-tiled subgraph with $\ceil{1/\eps} \leq k \leq \ceil{1/\eps}+1$.
			As every $K_4$-tiled graph on at least $\ceil{1/\eps} + 1$ vertices contains a $K_4$-tiled subgraph on either $\ceil{1/\eps}$ or $\ceil{1/\eps}+1$ vertices, the claim follows.
		\end{proof}

		\begin{claim}\label{clm:tiled-sparse}
			Asymptotically almost surely $\phi(H) \leq 7$ holds for every $H$ which is a $K_4$-tiled subgraph of $G \sim \mathbb{G}(n,p)$. 
		\end{claim}

		\begin{proof}
			Let $H$ be a $K_4$-tiled graph on at most $\ceil{1/\eps}$ vertices, satisfying $\phi(H) > 7$; equivalently, we have $5v(H) - 2e(H) \leq 0$. Then, the expected number of copies of $H$ in $G$ is at most 
			\begin{equation}\label{eq:expect-vanish}
				n^{v(H)} p^{e(H)}
				= n^{v(H) - (2/5+\eps)e(H)}
				= n^{\frac{1}{5} \cdot(5v(H) - 2e(H)) - \eps \cdot e(H)}
				\le n^{-6\eps} = o(1),
			\end{equation}
			where the above inequality holds since $H$ contains a copy of $K_4$ and thus $e(H) \geq 6$.
			This estimate, along with the fact that the number of graphs on at most $\ceil{1/\eps}$ vertices has order of magnitude $O_\eps(1)$, collectively imply that $G \sim \mathbb{G}(n,p)$ a.a.s.\ has the property that all $K_4$-tiled subgraphs $H$ of $G$ on at most $\ceil{1/\eps}$ vertices satisfy $5v(H) - 2e(H) \ge 1$. This property, together with Claim~\ref{clm:small-K_4-tiled}, completes the proof. 
		\end{proof}

		\begin{claim}\label{clm:intersection-1}
			Asymptotically almost surely $G \sim \mathbb{G}(n,p)$ does not have two edge-disjoint $K_4$-tiled subgraphs, $H_1$ and $H_2$, that satisfy $\phi(H_i) \ge 3$ for $i \in [2]$, and that have at least two vertices in common.
		\end{claim}

		\begin{proof}
			Suppose that $k := |V(H_1) \cap V(H_2)| \geq 2$ and set $H := H_1 \cup H_2$. Then, 
			\begin{equation} \label{eq::vHeH}
				v(H) = v(H_1) + v(H_2) - k \; \text{and}\; e(H) = e(H_1) + e(H_2). 
			\end{equation}
			As $\phi(H_i) \ge 3$, we have $5v(H_i) - 2e(H_i) \le 5$. It thus follow by~\eqref{eq::vHeH} that
			$$
				5v(H) -2e(H) = (5v(H_1) -2e(H_1)) + (5v(H_2) - 2e(H_2)) - 5k \leq 5 + 5 - 5 k \leq 0,
			$$
			where the last inequality holds since $k \geq 2$ by assumption.
			
			Following~\eqref{eq:expect-vanish}, the expected number of copies of $H$ in $G \sim \mathbb{G}(n,p)$ is at most 
			$$
				n^{\frac{1}{5} \cdot(5v(H) - 2e(H)) - \eps \cdot e(H)} \leq n^{- \eps \cdot e(H)} \leq n^{-6\eps} = o(1). 
			$$
			The claim now follows by a similar argument to that seen after~\eqref{eq:expect-vanish}.
		\end{proof}

		The following claim precludes long \emph{path} compositions of $K_4$-tiled graphs in $\C$.
		
		\begin{claim} \label{clm:K_4-path}
			Asymptotically almost surely $G \sim \mathbb{G}(n, p)$ does not have a collection of (pairwise) edge-disjoint $K_4$-tiled subgraphs, $H_1, \ldots, H_k$, with $k \ge \ceil{1/\eps}$, such that $\phi(H_i) \ge 3$ for every $i \in [k]$, and $|V(H_i) \cap V(H_{i+1})| = 1$ for every $i \in [k-1]$.
		\end{claim}

		\begin{proof}
			It suffices to prove the claim for $k = \ceil{1/\eps}$.
			Suppose that $H_1, \ldots, H_k$ is such a collection with $k = \ceil{1/\eps}$, and let $H = \bigcup_{i=1}^k H_i$. As $\phi(H_i) \ge 3$, or, equivalently, $5v(H_i) - 2e(H_i) \le 5$ for $i \in [k]$, it follows that
			\begin{equation*} 
				5v(H) - 2e(H) = \sum_{i = 1}^k(5v(H_i) - 2e(H_i)) - 5(k-1) \le 5. 
			\end{equation*}
			Following~\eqref{eq:expect-vanish}, the expected number of copies of $H$ in $\mathbb{G}(n,p)$ is at most 
			\begin{equation*}
				n^{\frac{1}{5}\cdot (5v(H) - 2e(H)) - \eps\cdot e(H)} \le n^{1 - \eps \cdot e(H)} \le n^{1 - 6\eps k} = o(1).
			\end{equation*}
			Since the number of possible such graphs $H$ is $O_{\eps}(1)$ (using Claim~\ref{clm:small-K_4-tiled}), the claim follows.
		\end{proof}
			
		The following claim precludes \emph{cyclic} compositions of $K_4$-tiled subgraphs in $\mathbb{G}(n,p)$.

		\begin{claim}\label{clm:K_4-collection}		
			Asymptotically almost surely $G \sim \mathbb{G}(n,p)$ does not have a collection of (pairwise) edge-disjoint $K_4$-tiled subgraphs, $H_1, \ldots, H_k$, such that $\phi(H_i) \ge 3$ for every $i \in [k]$, and $|V(H_{i}) \cap V(H_{i+1})| = 1$ for every $i \in [k]$ (with indices taken modulo $k$, i.e.,\ $H_k$ and $H_1$ share a vertex).
		\end{claim}

		\begin{proof}
			Suppose that $H_1, \ldots, H_k$ is such a collection, and let $H = \bigcup_{i=1}^k H_i$. Then $5v(H_i) - 2e(H_i) \le 5$ for $i \in [k]$, implying that
			\begin{equation}\label{eq:union-k}
				5v(H) - 2e(H) = \sum_{i = 1}^k (5v(H_i) - 2e(H_i)) - 5k \leq 0.
			\end{equation}
			As in previous claims, it follows that the expected number of copies of $H$ is $o(1)$. Since, by Claims~\ref{clm:K_4-path} and~\ref{clm:small-K_4-tiled}, the number of possible such graphs $H$ is $O_{\eps}(1)$, the claim follows.
		\end{proof}
		
		The following claim further restricts the paths of $K_4$-tiled subgraphs of $\mathbb{G}(n,p)$.


		\begin{claim}\label{clm:K_4-collection-path}		
			Asymptotically almost surely $G \sim \mathbb{G}(n,p)$ does not have a collection of (pairwise) edge-disjoint $K_4$-tiled subgraphs, $H_1, \ldots, H_k$, with $k \ge 2$, satisfying $|V(H_i) \cap V(H_{i+1})| = 1$ for every $i \in [k-1]$, such that 
			\begin{enumerate}[label = \rm(\roman*)]
				\item 
					$\phi(H_i) \ge 6$ for $i \in \{1,k\}$, and 
				\item
					$\phi(H_i) \ge 3$ for every $2 \leq i \leq k-1$.
			\end{enumerate}
		\end{claim}

		\begin{proof}
			Suppose that $H_1, \ldots, H_k$ is such a collection, and let $H = \cup_i H_i$. Then $5v(H_i) - 2e(H_i)$ is at most $2$ for $i \in \{1, k\}$ and at most $5$ for $2 \le i \le k-1$. Thus
			\begin{equation*} 
				5v(H) - 2e(H) = \sum_{i=1}^k (5v(H_i) - 2e(H_i)) - 5(k-1) \le 4 + 5(k-2) - 5k + 5 < 0.
			\end{equation*}
			The proof can be completed as in the proof of Claim~\ref{clm:K_4-collection}.
		\end{proof}

		Let $G \sim \mathbb{G}(n,p)$ satisfying all of the above properties (as stated in Claims~\ref{clm:small-K_4-tiled} -- \ref{clm:K_4-collection-path}) be fixed. 

		\begin{claim} \label{cl::colouringK4components}
			Let $H \in \C := \C(G)$ (that is, $H \subseteq G$ is a connected union of $K_4$-components $H'$ with $\phi(H') \ge 3$). Then $H$ admits a proper edge-colouring with all rainbow copies of $K_4$ sharing a common colour, say, red. 
		\end{claim}

		\begin{proof}[Proof of Claim~\ref{cl::colouringK4components} using Lemma~\ref{lem:phi}]
			Fix  $H \in \C$ and let $H_1, \ldots, H_k \in \HH$ be $K_4$-components satisfying $H = \bigcup_{i=1}^k H_i$. Without loss of generality, we may assume that $\phi(H_1) \ge \phi(H_i)$ for $2 \leq i \leq k$. Recall that, by the definition of $\C$, we have $\phi(H_i) \ge 3$ for every $i \in [k]$. 

			Let $F$ be the graph with vertex-set $[k]$ whose edges are pairs $ij$ such that $H_i$ and $H_j$ share a vertex; by Claim~\ref{clm:intersection-1}, such $H_i$ and $H_j$ share exactly one vertex. Observe that $F$ is connected by the definition of $\C$, and that $F$ is cycle-free by Claim~\ref{clm:K_4-collection}; in other words, $F$ is a tree.
			It follows that, upon appropriate relabelling (but keeping $H_1$ unchanged), we may insist on $[i]$ forming a subtree of $F$ for every $i \in [k]$, implying that $i$ is a leaf in this subtree for every $i \in [2, k]$. This means that $H_i$ has a common vertex with exactly one of $H_1, \ldots, H_{i-1}$, and so $H_i$ has a unique vertex in common with the graphs $H_1, \ldots, H_{i-1}$ for $i \in [2, k]$; denote this vertex by $u_i$. 
			Given $i \in [2, k]$, apply Claim~\ref{clm:K_4-collection-path} to the unique path in $F$ connecting $1$ and $i$. Noting that $\phi(H_i) \ge 3$ holds for every $i \in [k]$ by the definition of $\C$, it follows by the maximality of $\phi(H_1)$ that $\phi(H_i) \le 5$. 

			\medskip

			Recall that $\phi(H_1) \leq 7$ holds by Claim~\ref{clm:tiled-sparse}. Hence, one of the alternatives~\ref{itm:phi-2} and~\ref{itm:phi-3} of Lemma~\ref{lem:phi} must hold for $H_1$. Either way, it follows that $H_1$ admits a proper edge-colouring $\psi_1$ such that every copy of $K_4$ in $H_1$ which is rainbow under $\psi_1$ contains, say, a red edge. Moreover, as $3 \leq \phi(H_i) \leq 5$ holds for every $i \in [2,k]$, these graphs satisfy alternative~\ref{itm:phi-2} of Lemma~\ref{lem:phi}. Consequently, for every $i \in [2,k]$, there exists a triangle $T_i \subseteq H_i$ and a proper edge-colouring $\psi_i$ of $H_i$ such that every copy of $K_4$ in $H_i$ which is rainbow under $\psi_i$ contains $T_i$. We may assume that the colour sets used by $\psi_1, \ldots, \psi_k$ are pairwise disjoint. For every triangle $T_i$, let $v_i$ and $w_i$ be distinct vertices in $V(T_i) \setminus \{u_i\}$. Then $\{v_2 w_2, \ldots, v_kw_k\}$ forms a matching that does not meet $V(H_1)$. Recolour the edges of this matching red. The resulting colouring $\psi$ is a proper edge-colouring of $H$ such that every copy of $K_4$ in $H$ which is rainbow under $\psi$, contains a red edge, as required.
		\end{proof}

		We are finally ready to derive Proposition~\ref{thm:cover-rainbow-K4} from Lemma~\ref{lem:phi} and Claim~\ref{cl::colouringK4components}.
					
		\begin{proof}[Proof of Proposition~\ref{thm:cover-rainbow-K4} using Lemma~\ref{lem:phi}]
			Recall that the edge-set of $G$ can be decomposed into a collection $\HH:=\HH(G)$ of $K_4$-components, and a set $E'$ of edges of $G$ contained in no copy of $K_4$ in $G$. By Claim~\ref{cl::colouringK4components}, for every $H \in \C$, we can find a proper edge-colouring $\psi_H$ such that all rainbow copies of $K_4$ in $H$ contain, say, a red edge. Since the graphs in $\C$ are pairwise vertex-disjoint, the union of these colourings is a proper partial edge-colouring of $G$. Next, as every $H \in \HH$ which is not a subgraph of a member of $\C$ satisfies $\phi(H) \le 2$, every such $H$ satisfies alternative \ref{itm:phi-1} of Lemma~\ref{lem:phi}, i.e.,\ there is a proper edge-colouring $\phi_H$ of $H$ that admits no rainbow copies of $K_4$. 
			We assume that this proper edge-colouring uses colours unique to $H$. Finally, colour each of the edges of $G$ that are not contained in any copy of $K_4$ with a \emph{new} unique colour. 
			The union of all of these colourings results in a proper edge-colouring of $G$ with all rainbow copies of $K_4$ containing, say, a red edge, as required. 
		\end{proof}
	
	\subsection{Proof of Lemma~\ref{lem:phi}}\label{sec:classification}

		Throughout this section we will use the notation which was introduced in Section~\ref{sec:generating}.
		Let $H$ be a $K_4$-tiled graph satisfying $\phi := \phi(H) \leq 7$. In order to prove the lemma, we will construct a partial proper colouring $\psi$ of $H$ with the following property, where a copy of $K_4$ is called \emph{rainbow} if its edges that are coloured by $\psi$ have different colours: if $\phi \in \{0,1,2\}$ then there are no rainbow copies of $K_4$; if $\phi \in \{3, 4, 5\}$ then there is a triangle $T$ such that all rainbow copies of $K_4$ contain $T$; if $\phi \in \{6, 7\}$ then there is a matching $M$ of size at most $3$ such that all rainbow copies of $K_4$ contain an edge in $M$. Lemma~\ref{lem:phi} easily follows by extending $\psi$ to a full proper colouring of $H$, e.g.\ by colouring each uncoloured edge with a unique new colour.

		Let $H_0 \cong K_4, H_1, \ldots, H_r = H$ be a stretched sequence generating $H$. 
		The assumption $2\gamma+\beta = \phi \leq 7$ mandates that $0 \leq \gamma \leq 3$. We consider each possible value of $\gamma$ separately. 

		Having $\gamma = 0$ means that $H$ can be obtained from a copy of $K_4$ through a sequence of standard steps and vertex-steps without missing edges. Consequently, $H$ is $3$-degenerate (with the ordering of the vertices dictated by the steps of the stretched sequence generating $H$) and thus $K_5$-free. For higher values of $\gamma$, this cannot be assumed. Therefore, we add a fifth case to the case analysis over the possible values of $\gamma$ in which we consider the case that $H$ contains a copy of $K_5$. Dealing with it separately allows us to assume $K_5$-freeness throughout. 

		Before delving into the various cases, we state and prove the following observation, which controls the copies of $K_4$ introduced in certain steps.

		\begin{observation} \label{obs::unique-K4}
			Let $i \in [0, r]$.
			If the $i$th step is either a standard step or a vertex-step, then $\{x_i, y_i, z_i, w_i\}$ is the only copy of $K_4$ in $H_i$ that does not appear in $H_{i-1}$. 
			If the $i$th step is an edge-step with at least two missing edges, then the intersection of all the copies of $K_4$ in $H_i$ that do not appear in $H_{i-1}$ contains a triangle.
		\end{observation}

		\begin{proof}
			If the $i$th step is standard, then every copy of $K_4$ in $H_i$ which does not appear in $H_{i-1}$ contains at least one of $x_i$ and $y_i$, as all edges that are added in this step are incident with either $x_i$ or $y_i$. It is easy to see that the only such copy is spanned by $\{x_i, y_i, z_i, w_i\}$. Similarly, if the $i$th step is a vertex-step without missing edges, then all copies of $K_4$ in $H_i$ which are not in $H_{i-1}$ contains $x_i$ and are thus spanned by $\{x_i, y_i, z_i, w_i\}$. Finally, if the $i$th step is a vertex-step with missing edges, then by the assumption that the sequence is stretched, the graph $H_{i-1}'$, obtained from $H_{i-1}$ by adding all possible edges in $\{y_i, z_i, w_i\}$, does not contain a copy of $K_4$ which is not present in $H_{i-1}$. It follows that all new copies of $K_4$ in $H_i$ contain $x_i$. As above, the only such copy is $\{x_i, y_i, z_i, w_i\}$.

			If the $i$th step is an edge-step with at least two missing edges, denoted $e_i$ and $f_i$, then by definition of a stretched sequence, a step which adds one of $e_i$ and $f_i$ to $H_{i-1}$ does not create a new copy of $K_4$. It follows that all the copies of $K_4$ that appear in $H_i$ but not in $H_{i-1}$ contain both $e_i$ and $f_i$. In particular, their intersection contains a triangle.
		\end{proof}

		\subsection*{Case 1. $\gamma = 0$}

			As noted above, in this case $H$ can be obtained from a copy of $K_4$ through a sequence of standard steps and vertex-steps without missing edges. In particular, by Observation~\ref{obs::unique-K4}, the only copies of $K_4$ in $H$ are induced by the sets $\{x_i, y_i, z_i, w_i\}$. Moreover, $H$ is $3$-degenerate and thus $K_5$-free. 

			As the graph $H$ is being built, we \emph{partially} colour its edges so as to avoid a rainbow copy of $K_4$. The construction of this partial colouring can be seen in Figure~\ref{fig:partial-colouring}.
			
			\begin{figure}[ht!]
				\begin{framed}
					\begin{minipage}{.98\textwidth}
						{\bf Partial colouring procedure.}
						\begin{enumerate}[label = \rm \arabic*., ref = \rm \arabic*]
							\item \label{itm:colouring-H0} 
								In $H_0$, pick any matching of size $2$ and colour its edges with the same colour.

							\item \label{itm:colouring-standard}
								If the $i$th step is standard, colour $x_i z_i$ and $y_i w_i$ with the same new colour.

							\item \label{itm:colouring-vx}
								If the $i$th step is a vertex-step without missing edges, connecting $x_i$ to the triangle $T_i := y_i z_i w_i$, do the following.

								\begin{enumerate} [label = \rm (\roman*), ref = \rm 3(\roman*)]
									\item \label{itm:colouring-vx-old-colour}
										If there is an edge in $T_i$, say $y_i z_i$, coloured with a colour $\chi$ such that there is no edge of colour $\chi$ incident with $w_i$, colour the edge $x_i w_i$ with the colour $\chi$. If there is more than one way to do so, choose one arbitrarily.

									\item \label{itm:colouring-vx-new-colour}
										If the last step was impossible, but there is an edge of $T_i$, say $y_i z_i$, which is uncoloured, colour it and the edge $x_i w_i$ with the same \emph{new} colour.

									\item \label{itm:colouring-vx-problem}
										If steps~\ref{itm:colouring-vx-old-colour} and~\ref{itm:colouring-vx-new-colour} fail, mark $T_i$ as {\em problematic}, and move on to the next step (without colouring any edges).
								\end{enumerate}
						\end{enumerate}
					\end{minipage}
				\end{framed}
				\caption{Partial colouring avoiding rainbow $K_4$'s}
				\label{fig:partial-colouring}
			\end{figure}	
		
			Observe that at any point during the partial colouring procedure described in Figure~\ref{fig:partial-colouring}, the colouring is proper, and all copies of $K_4$ either contain two edges of the same colour (and will thus never become rainbow), or they contain a \emph{problematic triangle} (see Item \ref{itm:colouring-vx-problem} in Figure~\ref{fig:partial-colouring} for the definition).

			For a colour $\chi$ and a triangle $T$, we say that \emph{$\chi$ saturates $T$} (at a given moment with respect to a given partial colouring) if $T$ contains an edge of colour $\chi$ and the third vertex of $T$ (not incident with this edge) is also incident with an edge of colour $\chi$. The following claim plays a central role in proving Lemma~\ref{lem:phi} in the case $\gamma = 0$. 

			\begin{claim} \label{claim:problematic-triangles}
				If a triangle $T$ is problematic, then the sequence generating $H$ includes at least three vertex-steps in which a new vertex is attached to the triangle $T$.
			\end{claim}

			Before proving Claim~\ref{claim:problematic-triangles}, we show how to use it in order to conclude the proof of Lemma~\ref{lem:phi} in the case $\gamma = 0$.
			To this end, we consider three ranges of possible values of $\phi$, as specified by that lemma. 
			\begin{enumerate}
				\item
					If $\phi \in \{0, 1, 2\}$, then $\beta \leq 2$, and thus there are no problematic triangles, i.e., the partial colouring procedure described in Figure~\ref{fig:partial-colouring} can be extended to a proper edge-colouring of $H$ without rainbow copies of $K_4$.  

				\item
					If $\phi \in \{3,4,5\}$, then there is at most one problematic triangle. This implies the existence of a proper edge-colouring of $H$ and a triangle $T$ such that all rainbow copies of $K_4$ in $H$ contain $T$.

				\item 
					Finally, if $\phi \in \{6,7\}$, then there are at most two problematic triangles. This implies the existence of a proper edge-colouring of $H$ and two triangles $T_1$ and $T_2$ such that all rainbow copies of $K_4$ in $H$ contain either $T_1$ or $T_2$. It follows that there is a matching $M$ of size at most $2$ (consisting of one edge from each of the triangles $T_1$ and $T_2$) which meets every rainbow copy of $K_4$ in $H$.
			\end{enumerate}

			We now turn to the proof of the claim.

			\begin{proof}[Proof of Claim~\ref{claim:problematic-triangles}]
				We start with the following observation. 

				\begin{observation} \label{obs::chiNeighbourhood}
					For every colour $\chi$ and every vertex $u$, there is no point during the partial colouring procedure at which $u$ is not incident with an edge of colour $\chi$, yet there are two $\chi$-coloured edges in the neighbourhood of $u$.
				\end{observation}

				\begin{proof}
					Suppose for a contradiction that at some point there exist a vertex $u$ and a colour $\chi$ such that $u$ is not incident with an edge of colour $\chi$, yet there are two $\chi$-coloured edges, say, $ab$ and $cd$, such that $a, b, c$ and $d$ are in the neighbourhood of $u$. Upon its first appearance, $u$ has degree $3$. Since, moreover, $\gamma = 0$, at least one of the vertices in $\{a, b, c, d\}$ appears after $u$. Without loss of generality, assume that $d$ is the last vertex to appear amongst $\{u, a, b, c, d\}$. Since, moreover, $\gamma = 0$, the edge $cd$ appears either after $ab$ or at the same time. If $ab$ and $cd$ appear at the same time, then $d$ is added in a standard step, and $\{a, b, c, d\}$ is a copy of $K_4$, contrary to $H$ being $K_5$-free. So $cd$ appears after $ab$. We distinguish between the following two possible cases.
			
					\begin{enumerate}
						\item
							The edges $ab$ and $cd$ are coloured in the same step. By the description of the partial colouring procedure, $\{a,b,c,d\}$ forms a copy of $K_4$. But then $\{u,a,b,c,d\}$ forms a copy of $K_5$, contrary to $H$ being $K_5$-free.
					
						\item 
							The edges $ab$ and $cd$ are coloured in separate steps. By the description of the partial colouring procedure, a yet uncoloured existing edge can only be coloured with a new colour (i.e.\ one that did not appear previously) and each new colour may be used for at most one existing edge. Since $ab$ and $cd$ are both assigned the colour $\chi$, and $ab$ is an existing edge when $cd$ first appears, it must hold that $ab$ was coloured before $cd$ and the step in which $cd$ is coloured is a vertex-step attaching $d$ to a triangle $T$ containing a $\chi$-coloured edge. Since $d$ is adjacent to $u$ and $c$, and $\gamma = 0$, it follows that $T$ contains $u$ and $c$. Note that the $\chi$-coloured edge in $T$ does not include the vertex $c$ (because $cd$ is about to receive the colour $\chi$), which implies that $u$ is incident with a $\chi$-coloured edge, contrary to our assumption. \qedhere
					\end{enumerate}
				\end{proof}

				For a triangle $T$, let $k(T)$ denote the number of vertex-steps that attach a new vertex to $T$.  

				\begin{observation} \label{obs:saturated}
					The number of colours that saturate a triangle $T$ at any given moment is at most $k(T)+1$. 
				\end{observation}

				

				\begin{proof}
					The proof is by induction.
					Consider the first time that $T$ appears in the graph. Immediately before this moment, the graph contains at most two vertices of $T$, and when $T$ is added, edges of exactly one colour (new or old) are added. Hence, immediately after $T$ appears, it is saturated by at most one colour. Similarly, when a step that connects a new vertex to $T$ is performed, it increases the number of colours that saturate $T$ by at most $1$. 

					Suppose for a contradiction that there is a step that causes a new colour to saturate $T$, but which does not consist of connecting a new vertex to $T$. Note that such a step must reuse an old colour. Therefore, it is a vertex-step, and the resulting colouring is performed according to Item~\ref{itm:colouring-vx-old-colour} in the colouring procedure. More precisely, the step consists of connecting a new vertex $x$ to a triangle $yzw$, and without loss of generality, it colours the edge $xy$ with a colour $\chi$ that already appears on $zw$. By the assumption that after this step the colour $\chi$ saturates $T$, it follows that $y$ is a vertex in $T$ and the edge between the other two vertices of $T$ is coloured $\chi$. As this edge is distinct from $zw$ by the assumption that $yzw$ is not the triangle $T$, we find that before this step the neighbourhood of $y$ contains two edges of colour $\chi$, yet there is no edge of colour $\chi$ incident with $y$. This contradicts Observation~\ref{obs::chiNeighbourhood}. 

					To summarise, when it first appears, $T$ is saturated by at most one colour, and the number of colours that saturate $T$ can increase (by at most $1$) only via vertex-steps that connect a new vertex to $T$, as required.
				\end{proof}

				The proof of Claim~\ref{claim:problematic-triangles} follows easily from Observation~\ref{obs:saturated}. Indeed, a triangle $T$ is problematic if at some point there is a vertex-step attaching a new vertex to $T$, but $T$ is already saturated by three colours. It thus follows from Observation~\ref{obs:saturated} that there are at least three vertex-steps attaching a new vertex to $T$ (including the one which marks it problematic), as required for Claim~\ref{claim:problematic-triangles}.
			\end{proof}

		\subsection*{Case 2. $\gamma = 1$ and $H$ is $K_5$-free}

			Since $\gamma = 1$, there is an edge $xy$ which is introduced either through a vertex-step with a single missing edge or through a single $1$-edge-step. Either way, exactly one new copy of $K_4$ is created during this step. Indeed, this follows from Observation~\ref{obs::unique-K4} in the former case and from Claim~\ref{claim:first-edge-step} in the latter case.
			Denote the vertex set of this $K_4$-copy by $\{x,y,z,w\}$. We consider two cases, according to the type of step adding $xy$.

			\subsubsection*{Case 2a. $xy$ is introduced via a vertex-step with one missing edge}

				In this case $\gamma=1$ and $\beta \geq 1$, implying that $\phi = \beta + 2 \gamma \geq 3$. Suppose without loss of generality that the vertex-step introducing $xy$ attaches $z$ to $x, y, w$, where $xw$ and $yw$ are existing edges and $xy$ is a non-edge. 
				We apply the colouring procedure described in Figure~\ref{fig:partial-colouring}, with the additional rule that if the $i$th step is the vertex-step with a missing edge, then we colour $xy$ and $zw$ with the same new colour (note that the edges $xy$ and $zw$ do not appear in $H_{i-1}$ and so this step will be coloured successfully). 

				\begin{observation} \label{obs:successful-vx-steps}
					The first two vertex-steps are coloured successfully (if they exist).
				\end{observation}

				\begin{proof}
					We claim that before the second vertex-step, each triangle is saturated by at most two colours. If true, this readily implies the assertion, since, as noted above, the vertex-step with a missing edge is coloured successfully, and a vertex-step which attaches a vertex to an existing triangle which is saturated by at most two colours, is also guaranteed to be coloured successfully.

					To prove the aforementioned claim, note that any colour which was only used in a standard step, saturates only the triangles that are contained in the copy of $K_4$ introduced in this step. Therefore, any two distinct colours that were used only in standard steps, do not saturate the same triangle. Before the second vertex-step, at most one colour is used in a non-standard step. We conclude that every triangle is saturated by at most two colours, as claimed. 
				\end{proof}

				If $\phi \in \{3, 4, 5\}$, then there are at most three vertex-steps, the first two of which are coloured successfully by Observation~\ref{obs:successful-vx-steps}. Hence, there is at most one rainbow copy of $K_4$ (using Observation~\ref{obs::unique-K4}). If $\phi \in \{6, 7\}$, there are at most five vertex-steps, at most three of which are not coloured successfully (again using Observation~\ref{obs:successful-vx-steps}). It follows that there are at most three rainbow copies of $K_4$. It is easy to see that there is a matching $M$, of size at most $3$, that meets each of these copies. 

			\subsubsection*{Case 2b. $xy$ is introduced via an edge-step}

				By Claim~\ref{claim:first-edge-step}, this edge-step is preceded by at least two vertex-steps (without missing edges); in particular, we have $\phi = \beta + 2 \gamma \geq 4$. 

				Call a triangle $T$ \emph{dangerous} if there are at least two vertex-steps that attach a new vertex to $T$. We apply the colouring procedure described in Figure~\ref{fig:partial-colouring}, with the following modifications. 
				\begin{enumerate}
					\item
						If the $i$th step is a standard step, and one of $x_iz_i$ and $y_iw_i$ is $zw$, instead of colouring $x_i z_i$ and $y_i w_i$, colour $x_i w_i$ and $y_i z_i$ with the same new colour. In particular, standard steps never colour the edges $zw$.
					\item
						If the $i$th step is a vertex-step that attaches a new vertex $x_i$ to a triangle $T_i$, where $z, w \in \{x_i, y_i, z_i, w_i\}$, follow steps \ref{itm:colouring-vx-old-colour} and \ref{itm:colouring-vx-new-colour} but avoid colouring the edge $zw$. If this is impossible, do nothing.
					\item
						If the $i$th step is a vertex-step that joins a new vertex to a dangerous triangle $T$, do nothing.
					\item
						If the $i$th step is the edge-step, colour $xy$ and $zw$ with the same new colour. 
						Note that the previous modifications guarantee that $zw$ is uncoloured before this step.
				\end{enumerate}

				\begin{observation} \label{obs:successful-vx-steps-dangerous}
					The first vertex-step not extending a dangerous triangle will be coloured successfully. Moreover, all the vertex-steps that precede the edge-step and attach a vertex to a triangle which is not dangerous are successful.
				\end{observation}

				\begin{proof}
					The first part of the statement follows from the fact (which can be proved as in Observation~\ref{obs:successful-vx-steps}) that, before the second successful step which is not standard, each triangle is saturated by at most two colours.

					For the second part, note that Observation~\ref{obs:saturated} implies that, at any given moment before the edge-step, the number of colours that saturate a triangle $T$ is the number of vertex-steps that attach a new vertex to $T$ before the given moment, plus $1$. Thus, if a vertex-step that preceded the edge-step attaches a vertex to a triangle $T$ which is not dangerous, then $T$ is saturated by at most one colour, and so it can be coloured successfully while avoiding the colouring of $zw$. 
				\end{proof}

				Recall that the edge-step is preceded by at least two vertex-steps.
				If $\phi \in \{4, 5\}$, then there is at most one dangerous triangle. It follows from the observation above that either all unsuccessful vertex-steps attach a new vertex to the single dangerous triangle $T$, or there is at most one unsuccessful vertex-step. Either way, all rainbow copies of $K_4$ contain a fixed triangle.

				If $\phi \in \{6, 7\}$, then the number of dangerous triangles, denoted by $d$, is at most $2$. Observation~\ref{obs:successful-vx-steps-dangerous} implies that all but at most $3 - d$ rainbow copies of $K_4$ contain a dangerous triangle. It is not hard to conclude that there is a matching of size at most $3$ that covers all rainbow copies of $K_4$.

	\subsection*{Case 3. $\gamma = 2$ and $H$ is $K_5$-free}

		In this case $\phi \geq 4$. We follow the usual colouring procedure, described in Figure~\ref{fig:partial-colouring}, except that we do nothing during edge-steps, and in a vertex-step with at least one missing edge, where $x_i$ is attached to $y_i, z_i, w_i$ and $y_i z_i$ is missing, we colour $x_i w_i$ and $y_i z_i$ with the same new colour. As in Observation~\ref{obs:successful-vx-steps}, the first two vertex-steps are coloured successfully.
		We consider two subcases according to the number of steps used to introduce the two missing edges. 
		
	\begin{enumerate}
		\item 
			Assume first that both missing edges are introduced in the same step (either a vertex-step with two missing edges or a $2$-edge-step). It follows from Observation~\ref{obs::unique-K4} that there is a triangle $T$ which is contained by all the copies of $K_4$ that appear upon the introduction of these two edges.
			If $\phi \in \{4,5\}$, there is at most one vertex-step, which is coloured successfully as mentioned above, and thus all rainbow copies of $K_4$ contain the triangle $T$. If $\phi \in \{6, 7\}$, there is at most one vertex-step in which we fail to colour, implying that all rainbow copies of $K_4$ contain one of two given triangles, and so there is a matching of size $2$ that meets each of the rainbow copies of $K_4$.

		\item 
			Assume then that the missing edges are added in two separate steps.
			Consider first the case $\phi \in \{4,5\}$, which implies that $\beta \leq 1$. If there are two $1$-edge-steps, then the first one is preceded by at most one vertex-step without missing edges, contrary to the assertion of Claim~\ref{claim:first-edge-step}. Therefore, there must be one vertex-step with one missing edge, followed by a 1-edge-step. The vertex-step can be coloured successfully, thus, by Claim~\ref{claim:first-edge-step}, the intersection of all rainbow copies of $K_4$ contains a triangle.

			Next, consider the case $\phi \in \{6, 7\}$, which implies that $\beta \leq 3$. Suppose first that at least one of the missing edges is added in a vertex-step. Since the first two vertex-steps can be coloured successfully, we fail to colour in at most one vertex-step. It follows that all rainbow copies of $K_4$ contain either a given edge or a given triangle. It is easy to see that there is a matching $M$, of size at most $2$, that meets each of these copies. 
		
			We may thus assume that there are two $1$-edge-steps. Moreover, it follows by Claim~\ref{claim:first-edge-step} that the first 1-edge-step is preceded by two vertex-steps without missing edges and it creates exactly one new copy of $K_4$. Since the first two vertex-steps can be coloured successfully, we obtain a proper edge-colouring in which all rainbow copies of $K_4$ contain either one of two given copies of $K_4$ (stemming from the first edge-step and the third vertex-step, if it exists), or a given edge (stemming from the second edge-step). It is easy to see that there is a matching $M$, of size at most $3$, that meets each of these copies.  

	\end{enumerate}

	\subsection*{Case 4. $\gamma = 3$ and $H$ is $K_5$-free}

		In this case $\phi \geq 6$, implying that $\phi \in \{6,7\}$ and $\beta \leq 1$. We follow the modified colouring procedure described in Case 3; namely, we colour as in the usual colouring procedure, described in Figure~\ref{fig:partial-colouring}, but do nothing during edge-steps and colour vertex-steps with missing edges as described in the previous case. As explained above, the first two vertex-steps (if they exist) are guaranteed to be coloured successfully. We consider three subcases according to the number of steps used to introduce the three missing edges.
		
	\begin{enumerate}
		\item 
			Suppose, first, that all three missing edges are introduced in the same step. It follows from Observation~\ref{obs::unique-K4} that all copies of $K_4$ that appear upon the introduction of these edges intersect in a triangle. 
			There can be at most one additional non-standard step, which is a vertex-step that is coloured successfully. It follows that all rainbow copies of $K_4$ intersect in a triangle.
					
		\item 
			Next, suppose that the missing edges are introduced in two steps; in particular, one of these steps introduces two missing edges. As in the previous case, Observation~\ref{obs::unique-K4} implies that all copies of $K_4$ which appear upon the introduction of these edges intersect in a triangle. Recall that the only vertex-step, if it exists (recall that $\beta \le 1$) is coloured successfully. We thus obtain a proper edge-colouring such that all rainbow copies of $K_4$ contain a given triangle (stemming from the step with two missing edges, if it is an edge-step) or a given edge (stemming from the step with one missing edge, if it is an edge-step). It is easy to see that there is a matching $M$, of size at most $2$, that meets each of the rainbow copies of $K_4$. 

		\item 
			Finally, suppose that the missing edges are added in three separate steps. Since $\beta \leq 1$, it follows by Claim~\ref{claim:first-edge-step} that there is one vertex-step with a single missing edge followed by two $1$-edge-steps, and there is a triangle $T$ such that the copies of $K_4$ created by the first edge-step contain $T$. Since the vertex-step can be coloured successfully, we obtain a proper edge-colouring where all rainbow copies of $K_4$ contain a given triangle (stemming from the first 1-edge-step) or a given edge (stemming from the second 1-edge-step). It is easy to see that there is a matching $M$, of size at most $2$, that meets each of the rainbow copies of $K_4$. 
	\end{enumerate}

		\subsection*{Case 5. $H$ contains a copy of $K_5$}

			In this case, there is a sequence $K_5 \cong H'_0, H'_1, \ldots, H'_r = H$, such that $H'_i$ is obtained from $H'_{i-1}$ via a standard step, a vertex-step, or an edge-step. Strictly speaking, the aforementioned sequence is not a stretched sequence as it starts from a copy of $K_5$. Nevertheless, we do away with this technicality and assume that the sequence is \emph{optimal}, which, similarly to the notion of \emph{stretched}, means that $H_0' \cong K_5$, $H_i'$ is obtained from $H_{i-1}'$ via a standard, vertex, or edge step, and subject to these properties, the number of edges added between existing vertices is minimised, and the total number of steps is maximised. One can verify that all the results that were used to deduce Proposition~\ref{thm:cover-rainbow-K4} from Lemma~\ref{lem:phi} remain valid.
			
			Define $\alpha'$, $\beta'$ and $\gamma'$ analogously to the definition of $\alpha$, $\beta$ and $\gamma$, respectively. Then 
			$$
				v(H) = 5 + 2 \alpha' + \beta' \quad \text{and} \quad e(H) = 10 + 5 \alpha' + 3 \beta' + \gamma'.
			$$ 
			It follows that 
			\begin{equation} \label{eqn:phi}
				\phi = \phi(H) := 8 - 5v(H) + 2e(H) = 3 + \beta' + 2 \gamma' \geq 3.
			\end{equation}
			Observe that the definition of $\phi(H)$ is identical to the definition of $\phi$ introduced in Section~\ref{sec:proof-cover-rainbow-K4}, and so the assumption $\phi \leq 7$ seen in Lemma~\ref{lem:phi} remains valid. It thus suffices to consider values of $\beta'$ and $\gamma'$ for which $\beta' + 2 \gamma' \leq 4$ holds.

			We will make use of the following variant of Claim~\ref{claim:first-edge-step}~\ref{itm:edge-step-after-vx-step}; we provide only a sketch of its proof (as it is similar to the proof of Claim~\ref{claim:first-edge-step}) in Appendix~\ref{sec:proof-first-edge-step}. 

			\begin{claim} \label{claim:first-edge-step-K5}
				If the first edge-step in the sequence $K_5 \cong H'_0, H'_1, \ldots, H'_r = H$ is a $1$-edge-step, then it is preceded by at least one vertex-step.
			\end{claim}

			We follow a partial colouring procedure, similar to the one described in Figure~\ref{fig:partial-colouring}, with the following modifications. 
			\begin{enumerate}
				\item
					We replace Step \ref{itm:colouring-H0} in Figure~\ref{fig:partial-colouring} with the proper edge-colouring of $K_5$ described in Figure~\ref{fig:K5} (which admits no rainbow copies of $K_4$).

					\begin{figure}[ht]
						\centering
						\includegraphics[scale = 2]{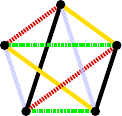}
						\caption{A proper edge-colouring of $K_5$ with no rainbow $K_4$'s}
						\label{fig:K5}
					\end{figure}
				\item
					For any standard step, we follow Step \ref{itm:colouring-standard} of the partial colouring procedure given in Figure~\ref{fig:partial-colouring}. 
				\item
					For any vertex-step, we follow Step \ref{itm:colouring-vx} of the partial colouring procedure given in Figure~\ref{fig:partial-colouring} with the following additional rule. If this vertex-step has missing edges, say it attaches $x$ to $y, z$ and $w$, and $yz$ is a non-edge which is added during this vertex-step, then we also allow the colouring of $xw$ and $yz$ with the same new colour. 
				\item
					We do not colour any edges during an edge-step.
			\end{enumerate}
			
			Note that, similarly to Observation~\ref{obs::unique-K4}, any standard step or vertex-step introduces a single new copy of $K_4$. Therefore, the above partial colouring guarantees that every $K_4$ which is coloured upon appearance will not be rainbow. 
			
			We will use the following claim, whose proof is similar to that of Observation~\ref{obs:saturated} above.

			\begin{observation} \label{obs:first-vertex-step} 
				The first vertex-step is coloured successfully, namely the copy of $K_4$ introduced by this step contains two edges of the same colour. 
			\end{observation}

			\begin{proof} 
				The assertion of the claim holds by Item 3 above for the first (or any other) vertex-step with missing edges. Hence, we only need to consider vertex-steps without missing edges. We monitor the number of colours that saturate each triangle. Triangles contained in $H_0'$ are initially saturated by two colours; triangles that appear following a standard step are initially saturated by one colour; and triangles that appear following an edge-step are saturated by at most two colours upon appearance (as they contain an edge that was previously missing and which is not coloured upon appearance). Therefore, since standard steps and edge-steps do not increase the number of colours saturating any existing triangle, immediately before the first vertex-step, every triangle is saturated by at most two colours. It follows that the first vertex-step is indeed coloured successfully.
			\end{proof}

			Recall that $3 \le \phi \le 7$ holds by \eqref{eqn:phi} and by the assertion of Lemma~\ref{lem:phi}.
			To complete the proof of Case 5, we consider the following five subcases. We will show that, if $\phi \in \{3,4,5\}$, the partial colouring procedure described above can be extended to a proper edge-colouring with at most one rainbow copy of $K_4$; and, if $\phi \in \{6, 7\}$, it can be extended to a proper edge-colouring such that all rainbow copies of $K_4$ can be covered by a matching of size at most $3$. 

			\begin{enumerate}
				\item
					$\phi \in \{3, 4, 5\}$ and $\gamma' = 0$. These values of $\phi$ and $\gamma'$ imply that $\beta' \leq 2$.\,\,
					Hence, there are at most two non-standard steps, all of which are vertex-steps with no missing edges. By Observation~\ref{obs:first-vertex-step}, the first of these steps is coloured successfully, so we end up with at most one rainbow $K_4$.
				\item
					$\phi \in \{3,4,5\}$ and $\gamma' = 1$. These values of $\phi$ and $\gamma'$ imply that $\beta' = 0$.\,\,
					Hence, there is one $1$-edge-step and no other non-standard steps, contrary to the assertion of Observation~\ref{claim:first-edge-step-K5}.
				\item
					$\phi \in \{6, 7\}$ and $\gamma' = 0$. These values of $\phi$ and $\gamma'$ imply that $\beta' \leq 4$.\,\,
					Hence, all non-standard steps are vertex-steps without missing edges, and there are at most four such steps. By Observation~\ref{obs:first-vertex-step}, the first such step is coloured successfully, implying that there are at most three rainbow copies of $K_4$ in $H$. It is easy to see that there is a matching $M$, of size at most $3$, that meets each of these copies.
				\item
					$\phi \in \{6, 7\}$ and $\gamma' = 1$. These values of $\phi$ and $\gamma'$ imply that $\beta' \leq 2$.\,\,
					Hence, either there is one vertex-step with one missing edge and at most one vertex-step without missing edges; or there are at most two vertex-steps without missing edges, and a single $1$-edge-step. Since the first vertex-step is coloured successfully by Observation~\ref{obs:first-vertex-step}, all but at most one rainbow copy of $K_4$ intersect in a given edge (namely, the edge added in the vertex-step with one missing edge in the former case, or the $1$-edge-step in the latter case, with the potential exceptional $K_4$ stemming from the second vertex-step; in the former case, there is in fact at most one rainbow $K_4$). Either way, it readily follows that there is a matching of size at most $2$ that meets all rainbow copies of $K_4$.  
					
				\item
					$\phi \in \{6, 7\}$ and $\gamma' = 2$. These values of $\phi$ and $\gamma'$ imply that $\beta' = 0$.\,\,
					It follows by Claim~\ref{claim:first-edge-step-K5} that the only non-standard step in this case is an edge-step with two missing edges. Therefore, all rainbow copies of $K_4$ intersect in an edge (in fact, they intersect in at least two adjacent edges due to the optimality of the sequence, implying that they intersect in a triangle). 
			\end{enumerate}

\subsection*{Acknowledgements}
	We would like to thank the anonymous referees for their many insightful and helpful comments.

\bibliographystyle{amsplain}
\bibliography{lit}

\appendix

\section{Emergence of small graphs in $\mathbb{G}(n,p)$} \label{sec:Janson}

	In this section we prove several claims that we used in previous sections regarding the appearance of fixed graphs in certain subgraphs of $\mathbb{G}(n, p)$.
	Throughout this section, we make repeated appeals to a result of Janson~\cite{Janson98} (see also~\cite[Theorem~2.18]{JLR}) regarding random variables of the form $X = \sum_{A \in \sS}I_A$. Here, $\sS$ is a family of non-empty subsets of some ground set $\Omega$ and $I_A$ is the indicator random variable for the event $A \subseteq \Omega_p$, where $\Omega_p$ is the so-called \emph{binomial random set} arising from including each element of $\Omega$ independently with probability $p$. For such random variables, set $\lambda := \Ex[X]$,  and define
	$$
	\Delta 
	:=  \frac{1}{2} \sum_{\stackrel{A,B \in \sS:}{A \neq B\text{ and } A \cap B \neq \emptyset}} \Ex[I_A I_B].
	$$
	The following result is commonly referred to as the \emph{probability of nonexistence} (see~\cite{JLR}). 

	\begin{theorem} \label{thm:Janson-exist} {\em~\cite[Theorem~2.18]{JLR}}
		For $X$, $\lambda$, and $\Delta$ as above we have
		$\Pro[X = 0] \leq \exp \left(-\frac{\lambda^2}{\lambda + 2\Delta}\right)$.
	\end{theorem}

	Of specific interest to us is the random variable $X_H:= X_H(n,p)$ which for a prescribed graph $H$ accounts for the number of (unlabelled) occurrences of $H$ in $\Gnp$. More specifically, for a prescribed $H$, let $\HH := \HH_n$ denote the family of (unlabelled) copies of $H$ in $K_n$. For every $\tilde H \in \HH$, let $Z_{\tilde H}$ denote the indicator random variable for the event $\tilde H \subseteq \Gnp$. Then, $X_H := \sum_{\tilde H \in \HH} Z_{\tilde H}$ counts the number of copies of $H$ in $\Gnp$. Note that 
	$$
	\Ex (X_H) 
	= \sum_{\tilde H \in \HH} p^{e(\tilde H)} 
	= \binom{n}{v(H)} \frac{(v(H))!}{|Aut(H)|} \cdot p^{e(H)} 
	= \Theta \left(n^{v(H)} p^{e(H)}\right),
	$$ 
	where $Aut(H)$ is the automorphism group of $H$. Writing $H_i \sim H_j$ whenever $(H_i,H_j) \in \HH \times \HH$ are distinct and not edge-disjoint, we define 
	\begin{align} 	
		\label{eq:Delta}
		\Delta(H) 
		& := \sum_{\substack{(H_i,H_j) \in \HH \times \HH \\ H_i \sim H_j}} \Ex[Z_{H_i} Z_{H_j}] = \sum_{\substack{(H_i, H_j) \in \HH \times \HH \\ H_i \sim H_j}} p^{e(H_i) + e(H_j) - e(H_i \cap H_j)} \nonumber \\
		& = \sum_{J \subsetneq H: \,\, e(J) \ge 1} \sum_{\substack{(H_i, H_j) \in \HH \times \HH \\ H_i \cap H_j \cong J}} p^{2e(H) - e(J)}  
		= O_H \left(n^{2v(H)} p^{2e(H)} \cdot \sum_{J \subsetneq H: \,\, e(J) \ge 1}  n^{-v(J)} p^{-e(J)} \right). 
	\end{align}

	Given a set $\C \subseteq \binom{[n]}{v(H)}$, we write $X_H(\C)$ to denote the number of copies of $H$ in $\Gnp$ supported on the members of $\C$, that is, 
	$$
	X_H(\C) = \{\tilde H \in \HH : V(\tilde H) \in \C \textrm{ and } \tilde H \subseteq \Gnp\}.
	$$ 
	Put
	\begin{equation}\label{eq:Delta-C}
		\Delta(H,\C) := \sum_{\substack{(H_i,H_j) \in \HH(\C) \times \HH(\C) \\ H_i \sim H_j}} \Ex[Z_{H_i} Z_{H_j}],
	\end{equation}
	where $\HH(C)$ serves as the analogue of $\HH$ for the copies of $H$ supported on $\C$. In particular, $\Delta(H, \C) \le \Delta(X_H)$. For $Y \subseteq [n]$, we abbreviate $X_H \left(\binom{Y}{v(H)} \right)$ to $X_H(Y)$ and $\Delta \left(H, \binom{Y}{v(H)}\right)$ to $\Delta(H,Y)$.

	\begin{corollary} \label{cor:Janson}
		Let $H$ be a graph, let $\eta > 0$ be fixed, and let $p = p(n)$.
		Suppose that $n^{v(J)} p^{e(J)} = \omega(1)$ for every induced subgraph $J \subseteq H$ that contains at least one edge.
		Let $\C \subseteq \binom{[n]}{v(H)}$ be a fixed family of size at least $\eta \binom{n}{v(H)}$. Then a.a.s.\ $G \sim \Gnp$ satisfies $X_H(\C) \ge 1$.
	\end{corollary}

	\begin{proof} 
		Write $\Delta := \Delta(H, \C)$ and $\lambda := \Ex[X_H(\C)]$. Then $\lambda = |\C| \cdot p^{e(H)} = \Theta(n^{v(H)} p^{e(H)}) = \omega(1)$. Moreover 
		$$
		\Delta \le \Delta(H) 
		= O_H \left(\left(n^{v(H)} p^{e(H)}\right)^2 \sum_{J \subsetneq H: \,\, e(J) \geq 1} n^{-v(J)} p^{-e(J)} \right)
		= o(\lambda^2)
		$$
		holds by \eqref{eq:Delta} (note that, by assumption, $n^{v(J)} p^{e(J)} = \omega(1)$ for every \emph{induced} subgraph $J$ of $H$ with at least one edge, but this implies that $n^{v(J)}p^{e(J)} = \omega(1)$ holds for \emph{every} subgraph $J$ of $H$ with at least one edge). It then follows by Theorem~\ref{thm:Janson-exist} that $\Pro[X = 0] \leq \exp \left(-\frac{\lambda^2}{\lambda + 2\Delta}\right) = o(1)$.   
	\end{proof}

	\begin{corollary} \label{cor:Janson2}
		Let $H$ be a graph, let $\eta > 0$ be fixed, and let $p = p(n)$.
		Suppose that $n^{v(J)} p^{e(J)} = \omega(n)$ for every induced subgraph $J \subseteq H$ that contains at least one edge.
		Then a.a.s.\ $X_H(Y) \geq 1$ holds for every subset $Y \subseteq [n]$ of size $|Y| \geq \eta n$.
	\end{corollary}

	\begin{proof} 
		Given $Y \subseteq [n]$ of size $|Y| \geq \eta n$, let $\lambda_Y := \Ex[X_H(Y)]$ and $\Delta_Y := \Delta(H, Y)$. Then $\lambda_Y = \Theta(n^{v(H)} p^{e(H)}) = \omega(n)$ and, by~\eqref{eq:Delta}, $\Delta_Y = o(\lambda_Y^2/n)$. It then follows by Theorem~\ref{thm:Janson-exist} that 
		$$
		\Pro[X = 0] \leq \exp \left(-\frac{\lambda_Y^2}{\lambda_Y + 2\Delta_Y}\right) = o(2^{-n}).
		$$ 
		The result follows by a union bound over all the choices of $Y \subseteq [n]$ of size $|Y| \geq \eta  n$.
	\end{proof}

	\medskip 

	In the remainder of this section, we use Corollaries~\ref{cor:Janson} and \ref{cor:Janson2} to prove Claims~\ref{cl::B2::Calc} to \ref{clm::C3::Calc}.

	\begin{claim} \label{cl::B2::Calc}
		Let $\eta > 0$ be fixed, let $p = p(n) = \omega(n^{-1})$. Then a.a.s.\ $X_{K_3}(\T) \geq 1$, where $\T \subseteq \binom{[n]}{3}$ is a prescribed fixed set of size $|\T| \geq \eta n^3$.  
	\end{claim}

	\begin{proof}
		By Corollary~\ref{cor:Janson}, it suffices to show that $n^{v(J)} p^{e(J)} = \omega(1)$ for every induced subgraph $J$ of $K_3$ containing at least one edge, that is, for $J \cong K_3$ and $J \cong K_2$. Recalling that $p = \omega(n^{-1})$, we observe that if $J \cong K_3$, then $n^{v(J)}p^{e(J)} = n^3 p^3 = \omega(1)$, and if $J \cong K_2$, then $n^{v(J)} p^{e(J)} = n^2 p = \omega(1)$; the claim readily follows.
	\end{proof}

	\begin{claim} \label{cl::B4::Calc}
		Let $\eta > 0$ be fixed, and let $p = p(n) = \omega(n^{-1})$. Then a.a.s.\ $X_{K_{1,4}}(Y) \geq 1$ for every $Y \subseteq [n]$ of size $|Y| \geq \eta n$. 
	\end{claim}

	\begin{proof}
		Let $J$ be an induced subgraph of $K_{1,4}$ with at least one edge, that is, $J \cong K_{1,r}$ for some $r \in [4]$. 
		Then, $n^{v(J)} p^{e(J)} = n^{r+1}p^{r} = \omega(n)$ and thus the claim follows by Corollary~\ref{cor:Janson2}.
	\end{proof}

	Let $R_7$ denote the graph obtained from $K_{1,2}$ by attaching two triangles to each of its edges, that is, $V(R_7) = \{u_1, u_2, u_3, w_1, w_2, w_3, w_4\}$ and 
	$$
	E(R_7) 
	= \{u_1 u_2, u_2 u_3, u_1 w_1, u_1 w_2, u_2 w_1, u_2 w_2, u_2 w_3, u_2 w_4, u_3 w_3, u_3 w_4\}.
	$$ 
	See Figure~\ref{fig:R} for an illustration.

	\begin{claim} \label{cl::31R7inAgoodSet::Calc}
		Let $\eta > 0$ and $k \in \mathbb{N}$ be fixed, let $R$ be the vertex-disjoint union of $k$ copies of $R_7$, and let $p = p(n) = \omega(n^{-2/3})$. 
		Let $\Z \subseteq \binom{[n]}{7k}$ be a fixed set of size $|\Z| \geq \eta n^{7k}$. Then a.a.s.\ $X_{R}(\Z) \geq 1$.  
	\end{claim}

	\begin{proof}
		A routine examination reveals that every subgraph of $R_7$ has average 
		degree strictly less than $3$. Consequently, every induced subgraph $J 
		\subseteq R$ with $e(J) \geq 1$ maintains this property; in particular, 
		$2e(J) < 3v(J)$. Thus, for any such $J$, it holds that  
		$$
		n^{v(J)}p^{e(J)} = \omega(n^{v(J) - (2/3)e(J)}) = \omega(1).
		$$ 
		Therefore, the claim follows by Corollary~\ref{cor:Janson}.
	\end{proof}

	\begin{remark}
		{\em The condition imposed on $p$ in Claim~\ref{cl::31R7inAgoodSet::Calc} can be mitigated to $p = \omega(n^{-7/10})$.} 
	\end{remark}

	Let 
	$$
	T_{k} = (\{x, v_1, \ldots, v_{2k}\}, \{x v_i : 1 \leq i \leq 2k\} \cup \{v_{2i-1} v_{2i} : 1 \leq i \leq k\})
	$$ 
	denote the graph obtained by \emph{gluing} $k$ edge-disjoint triangles along a single (central) vertex. See Figure~\ref{fig:T} for an illustration. 

	\begin{claim}\label{cl::T10onTheRight::Calc}
		Let $\eta > 0$ and $k \in \mathbb N$ be fixed, and let $p = p(n) = \omega(n^{-2/3})$.
		Then a.a.s.\ $X_{T_{k}}(Y) \geq 1$ for every $Y \subseteq [n]$ of size $|Y| \geq \eta n$. 
	\end{claim}

	\begin{proof}
		We claim that $v(J) - (2/3)e(J) \geq 1$ holds for every induced subgraph of 
		$T_k$. If $\delta(J) \geq 2$, then $J \cong T_\ell$ for some $\ell \in [k]$, 
		in which case $v(J) = 2\ell+1$ and $e(J) = 3\ell$ entailing the required inequality. 
		If $\delta(J) < 2$, repeatedly remove vertices of degree 
		at most $1$ until the remaining induced subgraph $J'$ consists of a single 
		vertex or satisfies $\delta(J') \geq 2$. Then $v(J') - (2/3)e(J') \ge 
		1$ holds for $J'$. 
		The subgraph $J$ can be obtained from $J'$ by repeatedly 
		adding vertices of degree at most $1$, and thus $v(J) - (2/3)e(J) \ge 1$ 
		holds as required.

		It thus follows that 
		$$
		n^{v(J)} p^{e(J)} = \omega(n^{v(J) - (2/3)e(J)}) = \omega(n)
		$$
		holds whenever $J$ is an induced subgraph of $T_{k}$ with 
		$e(J) \geq 1$. Therefore the claim follows by Corollary~\ref{cor:Janson2}.
	\end{proof}

	Let $\widehat{K}_{3,4}$ be the graph obtained from the complete bipartite graph $K_{3,4}$ by placing a triangle on its part of size $3$.

	\begin{claim} \label{cl::C1::Calc}
		Let $\eta > 0$ and $ k \in \mathbb{N}$ be fixed, let $K$ be the vertex-disjoint union of $k$ copies of $\widehat{K}_{3,4}$, and let $p = p(n) = \omega(n^{-7/15})$. Let $\Z \subseteq \binom{n}{7k}$ be a fixed set of size $|\Z| \geq \eta n^{7k}$. Then a.a.s.\ $X_K(\Z) \geq 1$.  
	\end{claim}

	\begin{proof}
		We claim that $15v(J) \geq 7e(J)$ holds whenever $J$ is an induced subgraph of 
		$K$ satisfying $e(J) \geq 1$. It suffices to prove this assertion for the induced 
		subgraphs of $\widehat{K}_{3,4}$. For the latter, suppose for a contradiction 
		that $J'$ is an induced subgraph of $\widehat{K}_{3,4}$ for which $15v(J') < 
		7e(J')$ holds. Then, the average degree of $J'$ is strictly larger than $4$. 
		This, in turn, implies that such a $J'$ satisfies $v(J') \geq 6$. There 
		are three non-isomorphic induced subgraphs of $\widehat{K}_{3,4}$ on at least $6$ vertices and
		it is easy to verify that all of them satisfy the aforementioned inequality,
		contrary to our assumption. 

		It follows that 
		$$
		n^{v(J)} p^{e(J)}
		= \omega(n^{v(J) - (7/15)e(J)})
		= \omega(1)
		$$
		holds for all induced subgraphs of $K$. Therefore, the claim 
		follows by Corollary~\ref{cor:Janson}.
	\end{proof}

	Let $K^{\Delta}_{1,25}$ denote the graph obtained from $K_{1,25}$ by attaching $49$ triangles to each of its edges, where the vertex not in $K_{1,25}$ is unique for every triangle. 

	\begin{claim}\label{clm::C3::Calc}
		Write $H = K^{\Delta}_{1,25}$.
		Let $\eta > 0$ be fixed and let $p = p(n) = \omega(n^{-7/15})$.
		Then a.a.s.\ $X_{H}(Y) \geq 1$ holds for every $Y \subseteq [n]$ of size $|Y| \geq \eta n$.  
	\end{claim}

	\begin{proof}
		We claim that $v(J) - (7/15)e(J) \geq 1$ holds whenever $J$ is an induced 
		subgraph of $H$ with at least one edge. 
		Suppose for a contradiction that the assertion is false and let $J'$ be a 
		minimal induced subgraph of $H$ with at least one vertex for which $v(J') - (7/15)e(J') < 1$ holds; note that in fact $v(J') > 1$.  
		Since $H$ is $2$-degenerate, $J'$ admits a vertex $u$ of degree at most $2$. 
		The graph $J'' := J' \setminus \{u\}$ satisfies 
		$$
		v(J'') - (7/15)e(J'') \le v(J') - 1 - (7/15)(e(J') - 2) = v(J') - (7/15)e(J') - 1/15 < 1
		$$
		contrary to the minimality of $J'$. 

		It thus follows that 
		$$
		n^{v(J)} p^{e(J)} 
		= \omega(n^{v(J) - (7/15)e(J)}) 
		= \omega(n)
		$$
		holds whenever $J$ is an induced subgraph of $H$ with $e(J) \geq 1$. Therefore, the 
		claim follows by Corollary~\ref{cor:Janson2}.
	\end{proof}

\section{Proof of Claims~\ref{claim:first-edge-step} and \ref{claim:first-edge-step-K5}} \label{sec:proof-first-edge-step}
	\begin{proof}[Proof of Claim~\ref{claim:first-edge-step}]
		Starting with Part~\ref{itm:edge-step-one-K4}, note that since the edge $xy$ is added in an edge-step, at least one of $x$ and $y$ does not belong to $H_0$. Up to relabelling, there are the following six options regarding the last step before all of $x,y,z,w$ appear in the graph: $x$ appears last (amongst $\{x,y,z,w\}$) in a standard step (together with some vertex $x' \notin \{y, z, w\}$); $x$ appears last in a vertex-step with no missing edges;  $z$ appears last in a standard step, along with another vertex $z' \notin \{x, y, w\}$; $z$ appears last in a vertex-step with no missing edges; $z$ and $w$ appear last in a standard step; or $x$ and $z$ appear last in a standard step. 
		The latter three options all imply that $x$ and $y$ are adjacent by the time the last of $x,y,z,w$ appears, contradicting the assumption that $xy$ is a non-edge at this point. (For instance, if $z$ appears last in a vertex-step with no missing edges, then it must be attached to a triangle consisting of the vertices $x,y,w$.) The third option implies that, upon appearance, $z$ has at most two neighbours in $\{x, y, w\}$, contradicting the assumptions that imply that all edges with both ends in $\{x, y, z, w\}$ except $xy$ are present at this point.
		Hence, we may assume that $x$ appears last amongst $\{x,y,z,w\}$. 

		Suppose for a contradiction that the addition of $xy$ completes two distinct copies of $K_4$, given by $\{x,y,z,w\}$ and $\{x,y,z',w'\}$; without loss of generality we may assume that $w \neq w'$. A similar argument to the one used above to establish that we may assume that $x$ appears last amongst $\{x,y,z,w\}$, can be used again so that we  may further assume that $x$ appears last amongst $\{x, y, z, z', w, w'\}$. It follows that, upon its appearance, $x$ has at most three neighbours amongst $\{w, w', z, z'\}$, implying that $z = z'$ and that $x$ is added in a vertex-step connecting it to $z$, $w$ and $w'$. Since there are no vertex-steps with missing edges and no edge-steps before $xy$, it follows that $\{x,y,z,w,w'\}$ forms a copy of $K_5$ (after $xy$ is added to the graph), contrary to the assumption that $H$ is $K_5$-free.

		\medskip

		Next, consider Part~\ref{itm:edge-step-intersect-triangle}. Let $F$ be the auxiliary graph with vertex-set $V(H) \setminus \{x, y\}$, where $zw$ is an edge of $F$ if and only if $\{x, y, z, w\}$ forms a copy of $K_4$ upon the appearance of the edge $xy$. Note that the conclusion of Part~\ref{itm:edge-step-intersect-triangle} is equivalent to the assertion that there is a vertex that meets all edges in $F$ (i.e.\ $F$ consists of a star and isolated vertices). It thus suffices to show that $F$ does not contain a triangle nor a matching of size $2$. Suppose first that the vertices $\{z, w, u\}$ form a triangle in $F$. It then follows that $\{x, y, z, w, u\}$ forms a copy of $K_5$ in $H$, contrary to the assumption that $H$ is $K_5$-free. Suppose then that $\{zw, uv\}$ is a matching of size $2$ in $F$. Since precisely one of the edges added before the first $1$-edge-step is added via a vertex-step with missing edges, we may assume that none of the edges with both endpoints in the set $A := \{x, y, z, w, u\}$ are the missing edge in a vertex-step with missing edges. Let $a \in A$ be a vertex that appears last among the vertices in $A$ (possibly, along with another vertex of $A$). Then, the neighbourhood of $a$ in $A$ upon the appearance of $a$, forms a clique. If $a \in \{x, y\}$, this means that $\{z, w, u\}$ forms a triangle, implying that $A$ induces a copy of $K_5$ in $H$, a contradiction. If $a \in \{z, w, u\}$, then $xy$ is an edge upon the appearance of $a$, contrary to the assumption that $xy$ was added during an edge-step.

		\medskip

		For the proof of Part~\ref{itm:edge-step-after-vx-step}, we may assume that no vertex-steps with missing edges precede the first edge-step; otherwise the assertion holds trivially. As in the proof of Part~\ref{itm:edge-step-one-K4}, we may again assume that $x$ appears after $y, z$, and $w$. The vertex $x$ appears either in a standard step or in a vertex-step (with no missing edges). If the former occurs, then in this step $x$ and some vertex $x' \notin \{y, z, w\}$ are added to the graph and are connected to each other and to $z$ and $w$.  We may then replace the latter step and the edge-step in which $xy$ is added to the graph with two consecutive vertex-steps with no missing edges: the first attaching $x$ to $\{y,z,w\}$ and the second attaching $x'$ to $\{x,z,w\}$. This results in a smaller value of $\gamma$, contradicting the minimality (as stated in~\ref{itm:min-gamma}) of the stretched sequence generating $H$.

		Suppose then that $x$ appears last in a vertex-step (with no missing edges). Therefore, in this step $x$ is attached to a triangle, spanned by $\{x',z,w\}$ for some $x' \neq y$. 
		
		Assume first that $y$ appears after $x', z$ and $w$ have all appeared. Then, the order of the steps can be altered so that the step adding $x$ is performed immediately after the appearance of $x', z$, and $w$. This means that $y$ is then the last vertex to appear amongst $\{x,y,z,w\}$. If $y$ appears in a standard step, we again obtain a contradiction to the minimality of $\gamma$ of the stretched generating sequence, as seen in the paragraph before last. Otherwise, $y$ appears in a vertex-step, implying that there are at least two vertex-steps before the first edge-step. This concludes the proof in this case, as the two vertex-steps (which add $x$ and $y$) precede the first edge-step also in the original sequence. 

		Assume then that at least one of $\{x',z,w\}$ does not appear before $y$. If all of $\{x', y, z, w\}$ appear together, they belong to $H_0$ and thus form a $K_4$; together with $x$ they thus eventually form a $K_5$, contrary to the assumption that the graph is $K_5$-free. Similarly, if $z$ or $w$ appear last amongst $\{x', z, w\}$ (possibly together with another of these three vertices), then $y$ and $x'$ must be adjacent, again implying that $\{x',x,y,z,w\}$ forms a $K_5$ in $H$. It follows that $x'$ appears after $y, z, w$. If $x'$ appears in a vertex-step, then there are at least two vertex-steps before the first edge-step, as required. If $x'$ appears in a standard step, then it is added together with a vertex $x''$ and they are both connected to one another and to $z$ and $w$. We can then modify the sequence as follows: immediately after $\{y,z,w\}$ all appear, attach $x$ to $\{y,z,w\}$, then attach $x'$ to $\{x,z,w\}$, and then attach $x''$ to $\{x',z,w\}$; this decreases $\gamma$, contradicting the minimality of the stretched generating sequence of $H$.
	\end{proof}
	
	\begin{proof}[Proof of Claim~\ref{claim:first-edge-step-K5}]
		The proof is similar to the proof of Claim~\ref{claim:first-edge-step}, provided above; therefore we provide only a sketch. Denote the edge that is added in this first edge-step by $xy$, and suppose that it completes a $K_4$ whose vertex set is $\{x, y, z, w\}$. As in the proof of Claim~\ref{claim:first-edge-step}, without loss of generality, we may assume that $x$ appears after $y,z,w$. Moreover, if it appears in a standard step along with another vertex $x'$ (where $x' \notin \{y, z, w\}$), we show that the sequence $H_0', \ldots, H_r'$ could be rearranged in such a way that, instead of the standard step adding $x$ and $x'$ and the edge step adding $xy$, first $x$ and then $x'$ are added in vertex steps with no missing edges. This is a contradiction to the optimality of the sequence $H_0', \ldots, H_r'$ (here we retain the notation of the original setting in which Claim~\ref{claim:first-edge-step-K5} is stated). This implies that $x$ appears in a vertex-step, as required.
	\end{proof}

\end{document}